\documentclass[11pt,reqno]{article}

\usepackage{psfrag, graphicx, amsmath, amssymb, amsthm, fancyhdr, marvosym, yhmath, appendix, color, mathtools, multirow, epstopdf, color}

\newcommand{\Z}{\mathbb{Z}}

\newcommand{\R}{\mathbb{R}}

\newcommand{\C}{\mathbb{C}}
\newcommand{\F}{\mathcal{F}}
\newcommand{\X}{\mathcal{X}}

\newcommand{\FF}{\mathcal{F}}
\newcommand{\RR}{\mathbb{R}}

\numberwithin{equation}{section}
\newtheorem{thm}{Theorem}[section]
\newtheorem{defn}[thm]{Definition}
\newtheorem{prop}[thm]{Proposition}
\newtheorem{lemma}[thm]{Lemma}

\date{}
\begin{document}

\title{Nonhomogeneous  Boundary Value Problems of Nonlinear Schr\"odinger Equations in a Half Plane\footnotetext{2010 {\it Mathematics Subject Classification}.
Primary 35Q55.}}

\author{ Yu Ran \\
{\small  Department of Mathematics} \\ {\small China Jiliang University}
\\ {\small
258 Xueyuan Road, Hangzhou 310018, China}
\\
\\Shu-Ming  Sun
\\ {\small  Department of Mathematics} \\ {\small Virginia
Polytechnic Institute and State University}
\\ {\small
Blacksburg, Virginia 24061, USA}\\
\\
Bing-Yu Zhang \\ {\small Department of
Mathematical Sciences} \\ {\small University of Cincinnati}
\\ {\small Cincinnati,
Ohio 45221, USA}
}

\baselineskip=14.5pt

\maketitle



\begin{abstract}
This paper discusses the initial-boundary-value problems (IBVPs)  of  nonlinear Schr\"odinger equations posed  in a half plane $\RR\times \RR^+$ with nonhomogeneous
Dirichlet boundary conditions. For any given $s \ge 0$,
if the initial data $\varphi (x, y)$ are in Sobolev space $H^s(\RR\times \R^+) $
with the boundary data $ h ( x, t) $ in an optimal space ${\cal H}^s(0,T)$  as defined in the introduction, which is slightly weaker than the space $$H^{(2s+1)/4}_{t} \big ([0, T]; L_x^2(\RR ) \big ) \cap L^2_t \big ( [ 0, T]; H^{s+ 1/2} _x ( \RR ) \big ),$$ the local well-posedness of the IBVP in
$ C ( [0, T] ; H^s ( \RR\times \RR^+ ) )$ is proved.
The global well-posedness is also discussed for $s = 1$.
The main idea of the proof is to derive a boundary integral operator for the corresponding nonhomogeneous boundary
condition and obtain the Strichartz estimates for this operator. The results presented in the paper hold also  for  the IBVP posed in a half
space $ \RR^{n}\times \RR^+$ with any $n>1$.

\end{abstract}

\allowdisplaybreaks

\section{Introduction}

\setcounter{equation}{0}

The paper concerns an
initial-boundary-value problems (IBVP) of the nonlinear Schr\"odinger
(NLS) equation posed in a half plane $( x, y)\in \mathbb{R} \times \mathbb{R}^+$ (or a half space $( x, y)\in \mathbb{R}^n \times \mathbb{R}^+$),
\begin{equation}\label{1.1}
\left \{ \begin{array}{l}
           iu_t + u_{xx} + u_{yy}  + \lambda |u|^{p-2} u =0, \qquad x\in \mathbb{R}, y \in \mathbb{R}^+ \, ,  \ t\in \mathbb{R
           },\\
           u(x,y, 0) =\varphi (x, y ), \qquad
           u(x, 0,t) =h (x, t),
         \end{array} \right.
         \end{equation}
where $p\geq 3$ is a  constant (note that we only study the case for $p \geq 3$ although many of the results in the paper still hold for $p > 2$; also in $\RR^n \times \RR^+$, $u_{xx}$ is replaced by the Laplacian $\Delta_x u$ in $\RR^n$). The NLS equation has many applications and
 is derived from various research areas ranging from theory of optics to theory of water waves. In particular,
the NLS equation is a  popular model used recently to explain the formation of rogue waves observed in sea or oceans
 \cite{CHA2011, pere1983}. Our motivation to study the problem  posed in a half plane originated from the water-wave problems. Let the free surface occupy the region $\{ (x, y ) \ | \ y > 0\} $ and a wave-maker be placed at $y = 0$. To predict the wave motion generated by the wave-maker, it is necessary to specify the initial and boundary conditions that make the surface waves propagate into the region $ y > 0$, which gives an initial boundary value problem in a half plane $ y > 0$.

Here, we only consider the mathematical problem of the NLS equations (\ref{1.1})  and  concentrate
on its  \emph{well-posedness} in the classical Sobolev space $H^s(\mathbb{R} \times \mathbb{R}^+)$ (or $H^s(\mathbb{R}^n \times \mathbb{R}^+)$).
If $(x, y ) \in \mathbb{R}^n \times \mathbb{R}$, i.e., $x\in \mathbb{R}^n$ for simplicity, this is a
pure initial value problem (IVP),  i.e.,
\begin{equation}
 \label{1.2}
 iu_t + \Delta_x u  +\lambda |u|^{p-2}u =0, \  u(x,0)=\varphi (x), \qquad x \in \mathbb{R}^n, \ t\in \mathbb{R}\, .
\end{equation}
In the past several decades, the IVPs of
the NLS equations posed on  $\mathbb{R}^{n}$
have been extensively studied for their  well-posedness in the space  $H^s (\mathbb{R}^n )$ (cf. \cite{bourgain-1,
bourgain-2, cazenave, CFH2011,
caz-weiss, gini-velo,kato-1,kato-2,tsu-1, tsu-4} and the
references therein). In particular, significant progresses have been made for
the well-posedness of the problem with lower regularity.
In contrast to the IVP of the NLS equation, the
study of the IBVP (\ref{1.1}) with
nonhomogeneous boundary condition has fallen behind (cf.
\cite{brezis,bu1994,bu2000, bu2001,
bu2005,bc,holmer,kam,strauss,tsu-1,tsu-2,tsu-3} and the references
therein). The low regularity feature for the IBVP (\ref{1.1}) without $x$ (i.e., in a half line) was first studied by Holmer \cite{holmer}. He showed that for $0 \leq s < 1/2$ with $ 3\leq p < \frac{6-2s}{1-2s}$ or for $ 1/2 < s < 3/2$ with $ 3 \leq p < \infty$, there exists a $T > 0 $ such that if $\varphi \in H^s ( \RR^+ ) $ and $ h \in H^{\frac{2s+1}{4}}_{loc} (\RR^+ )$, then the IBVP (\ref{1.1}) in a half line
has a solution in $C([0, T]; H^s (\RR^+ ))$. The well-posedness of the IBVP (\ref{1.1}) in a half line  with solutions
in the space $C([0, T]; H^s (\RR^+ ))$ for $s \geq 0$ is also discussed in \cite{Sun_NSE_1d} using boundary integral  operator method (see also \cite{ET} for the  recent  study  of the IBVP (\ref{1.1}) in a half line).  For higher dimensional cases, Audiard \cite{Audiard-1} investigated the non-homogeneous boundary value problem for some general linear dispersive equations and obtained a-priori estimates and well-posedness for the pure boundary value problems
associated to such class of equations or the IBVPs of the linear Schr\"odinger
equations in a half-space, while \cite{Audiard-2, Audiard-3} study the corresponding IBVPs for the nonlinear Schr\"odinger
equations in spatial domains that are the exterior of non-trapping compact or star-shaped obstacles. Here, we remark that the results obtained here are for the IBVPs of nonlinear Schr\"odinger
equations in a half-space with optimal boundary data and the method used is the boundary integral operator method, which are totally different from those in \cite{Audiard-1, Audiard-2, Audiard-3}.

In this paper, we study the IBVP (\ref{1.1}) for its local and global
well-posedness. Here, the same idea and method can be used to study  (\ref{1.1}) for $ x\in \RR^n$.
Our goal is to advance the study of the IBVP (\ref{1.1}) to the same level as that for the IVP (\ref{1.2}).
 In order to have the solution of (\ref{1.1}) in  the space $C([0, T]; H^s (\RR\times \RR^+ ))$ with $s \geq 0$,  while the initial  the initial  value $\varphi ( x, y) $ is chosen from  $H^s (\RR\times \RR^+ )$,  the  optimal choice of the function space for the boundary data $h ( x, t)$ needs some discussion.  Based upon the scaling argument,   it seems natural  (see \cite{Audiard-1,Audiard-2}) to choose  $h$ from the space
\begin{equation}\label{1.2.1}
 {\cal W}^s (0,T):=H^{\frac{2s+1}{4}}_{t} \left (0,T; L^2_x (\R)\right )\cap L^2 _{t} \left (0,T; H^{s+\frac12}_x(\R)\right ).
\end{equation}
It turns out, however,
 that the space ${\cal W}^s (0,T)$  is not the optimal choice. Indeed, as  the trace of the solution  $v$ of the linear Schr\"odinger equation in $\RR^2$,
$$
i v_t + v_{xx} + v_{yy} = 0 \, ,\quad v ( x, y, 0 ) = \varphi ( x, y) , \quad (x, y ) \in \RR^2 , \quad t \in \RR\, ,
$$
 most likely does not belong to the space ${\cal W}^s (0,T)$  (see the discussion in Section 2),
we will show that the solution $v$  possesses the  following
 trace property (see Lemma \ref{lemm2.2} in Section 2):

\smallskip
\emph{If $\varphi \in H^s (\RR^2) $, then  the trace of the solution $v\in C(\R;H^s(\RR^2))$,  $ v_b ( x, t) = v (x, 0, t) $,  belongs to the
space
$$
{\cal H}^s (\R^2) :=  \left \{ w\in L^2 (\R^2): \quad  \big ( 1+ | \lambda| + |\xi| ^2 |\big )^{s /2}\big | \lambda + |\xi| ^2 \big |^{1/4} \FF[ w] ( \lambda, \xi ) \in L^2(\RR^2 )\right \}
$$
with
\[ \|w\|_{{\cal H}^s(\R^2)}:= \left \| \big ( 1+ | \lambda| + |\xi| ^2 |\big )^{s /2}\big | \lambda + |\xi| ^2 \big |^{1/4} \FF[ w] ( \lambda, \xi ) \right \|_{L^2(\RR^2 )},\]
  where  $\FF[ w]$ stands for the Fourier transform of $w$ with respect to both $x$ and $t$.}
   The choose of $h$ from ${\cal H}^s (\R^2)$ is optimal and the following definition of the well-posedness for the IBVP (\ref{1.1}) is  then natural.

\begin{defn} \label{defi1.1}(Well-posedness) For any given $ s \in \RR, T > 0$, the IBVP (\ref{1.1}) is locally well-posed in $H^s ( \RR \times \RR^+ ) $
if for any constant $ \mu > 0$ there is a $ T^* \in (0, T ] $ such that for $\varphi \in H^s ( \RR \times \RR^+ ) $ and $ h \in {\cal H}^s (0,T)$
satisfying
$$
\| \varphi \|_{ H^s ( \RR \times \RR^+ )} +  \|  h \|_{{ \cal H}^s (0,T)} \leq \mu
$$
and some compatibility conditions, the IBVP (\ref{1.1}) has a unique solution in $C ( [0, T^*]; H^s ( \RR \times \RR^+ ))$, which continuously depends upon $( \varphi , h)$ in the corresponding spaces. If  $T^*$ can be chosen independent of $r$, then the IBVP (\ref{1.1}) is globally well-posed. Here, ${\cal H}^s(0,T)$ denotes the restriction space of ${\cal H}^s(\R^2)$ to the domain $\R\times (0,T)$.
\end{defn}

For small $s\geq 0 $, we need to address the meaning of solutions of the IBVP (\ref{1.1})
satisfying the initial and boundary conditions. We use the following definition \cite{Sun_CUC_Evo}.

\begin{defn} \label{defi1.2}For given $ s < 2$ and $T > 0$, let  $\varphi \in H^s ( \RR \times \RR^+ ) $ and $ h \in{ \cal H}^s (0,T)$.  Then $u (x, y, t) $ is called a mild solution of  the IBVP (\ref{1.1}) if there is a sequence
$$
u_n \in C \left ( [0, T]; H^2 ( \RR \times \RR^+ )\right ) \cap C^1 \left ( [0, T]; L^2 ( \RR \times \RR^+ )\right )\, ,\quad n = 1,2,\dots ,
$$
satisfying the following properties:
\begin{enumerate}
\item $u _n $ is a solution of (\ref{1.1}) in $L^ 2 ( \RR \times \RR^+ )$ for $ 0 \leq t \leq T$;
\item $u_n\rightarrow u $ in $C ( [0, T]; H^s ( \RR \times \RR^+ ))$ as $n \rightarrow \infty$;
\item as $n\rightarrow \infty$, $u_n ( x, y , 0) \rightarrow \varphi (x, y)$ in $H^s ( \RR \times \RR^+ )$ and
$
u _n ( x, 0 , t) \rightarrow h ( x, t)\ \mbox{in}\ {\cal H}^s (0,T)\, .
$
\end{enumerate}
\end{defn}

To study the well-posedness of the IBVP (\ref{1.1}), we first note that if $h = 0$, i.e.,
the homogeneous boundary condition, then the problem can be reduced to a special case of the IVP (\ref{1.2}) using an odd extension of the solution to $ y < 0$. Thus, the well-posedness result of (\ref{1.1}) with $h =0$ follows from the known result for (\ref{1.2}).
In this paper, we consider the IBVP of (\ref{1.1}) with $ h \not \equiv 0$.
We say that $(q, r) $ is an admissible pair if $ (1/q) + ( 1/r) = (1/2)$. Also, denote
$W^{s, p}_{xy} ( \RR \times \RR^+ ) $ as the classical Sobolev space in $L^p$-norm.
Thus, $W^{ s, 2 }_{xy}  = H^s_{xy}$.
The main result obtained in this paper can be summarized as follows.

\begin{thm}\label{theo1.3}
For given $ s \geq 0$, $T > 0$ and $\mu >0$, assume $\varphi \in H^s (\R\times \R^+) $ and $h\in {\cal H}^s(0,T)$ satisfying
$$
\| \varphi \|_{ H^s ( \RR \times \RR^+ )} +  \|  h \|_{ {\cal H}^s (0,T)} \leq \mu
$$
and  some natural  compatibility conditions.
\begin{itemize}
\item[(i)]
For $0 \le s <1$ and $ 3 \leq p < (4-2s)/(1-s) $, there exist a $T^*>0$ and a suitable admissible pair $(q,r)$ such that (\ref{1.1}) is locally well-posed for
$$
u \in C_t \left([0,T^*]; \, H^s_{xy}(\R \times \R^+)\right) \bigcap L^q_t \left([0,T^*]; \, W^{s,r}_{xy}(\R \times \R^+)\right) \, .
$$
Moreover, if $ p = (4-2s)/(1-s)$, $\mu $ must be small.
\item[(ii)]
For $s=1$, there exists a $ T^* >0$ such that  for any admissible pair $( q, r)$ with $ 2 < q < \infty$, (\ref{1.1}) is locally well-posed for $u \in  C_t \left([0,T^*]; \, H_{xy}^1(\R \times \R^+)\right) \bigcap L^q_t \left([0,T^*]; \, W^{1,r}_{xy}(\R \times \R^+)\right)$.
\item[(iii)]
If $s > 1$ and $3 \le p < \infty$ (assume $1<s \le p-1$ for $s\in\Z$ or $1<[s]\le p-2$ for $s\notin\Z$  when $p$ is not an even integer), there exists a $ T^* >0$ such that (\ref{1.1}) is locally well-posed for $u \in  C_t \left([0,T^*]; \, H^s_{xy}(\R \times \R^+)\right)$, where
$[s]$ is the largest integer that is less than or equal to $s$.
\end{itemize}
\end{thm}

Here, we note that for $0\leq s \leq 1$, the local well-posedness
result presented in  Theorem \ref{theo1.3} is conditional since the condition on $u$ with
\begin{equation} \label{x-1.1}u\in L^q_t \left([0,T^*]; \, W^{s,r}_{xy}(\R \times \R^+)\right)\end{equation}
is needed to guarantee
the uniqueness. We may ask whether it is possible to remove the condition (\ref{x-1.1}).
If the condition (\ref{x-1.1})  can be removed, the corresponding well-posedness
will be called unconditional.
\begin{thm} {\rm (unconditional well-posedness)}  \label{theo1.4}
If $0\leq s\leq 1$ is given, then the condition (\ref{x-1.1})
can be removed, and therefore the  well-posedness is unconditional.
\end{thm}
By \cite{Sun_CUC_Evo}, Theorems \ref{theo1.3} and \ref{theo1.4}  show that the solution obtained is a mild
solution defined in Definition \ref{defi1.2}.
Next, we state the global well-posedness result
for (\ref{1.1}).
\begin{thm}\label{theo1.5}
Assume that either
$p \geq 3$ if $\lambda < 0$ or $3 \leq p \leq (10/3)$ if $  \lambda >
0.$ The IBVP (\ref{1.1}) is globally well-posed in
$H^1 (\RR\times \R^+)$ with
$\varphi \in H^1 (\mathbb{R}\times \mathbb{R}^+)$ and $ h\in H^1_{t, loc} \left ( \RR ; L_x^2 (\RR )\right  ) \cap L_t ^2 \left ( \RR; H^{\frac{3}{2}}_{x} (\mathbb{R})\right )$.
\end{thm}

The idea of proof for the local well-posedness is based on the method introduced in \cite{Sun_KdV_qp,co-kenig}
for studying the IBVP of the Korteweg-de Vries equation (see also \cite{holmer} for  the study of the IBVP of the nonlinear Schr\"odinger equation).
The key step is to study  the following nonhomogeneous boundary value problem:
\begin{equation}\label{1.1-1}
\left \{ \begin{array}{l}
           iu_t + u_{xx} + u_{yy}  =0, \qquad x\in \mathbb{R}, \quad y \in \mathbb{R}^+ \, ,  \quad  t\in \mathbb{R},\\
           u(x,y, 0) =0, \qquad
           u(x, 0,t) =h (x, t).
         \end{array} \right.
         \end{equation}
Applying the Laplace transform to the linear equation  with respect to $t$  leads us to an explicit integral representation of the solution $u$ in terms
of the boundary data $h$, called  the boundary integral operator,
\[ u(x,y,t):= [W_{b}h](x,y,t) .\]
It will be shown that for any given $s\geq 0$, $T>0$  and $h\in {\cal H}^s(0,T)$,  the IBVP (\ref{1.1-1}) admits a unique solution $u\in C([0,T]; H^s(\R\times \R^+))$ and, moreover,
\begin{equation}\label{1.1-2}  \| u\|_{L^q (0,T; W^{s,r} (\R\times \R^+))} \leq C\| h\|_{{\cal H}^s (0,T)} \end{equation}
for any admissible  pair $(q, r)$ with $1/q+1/r=1/2$, $2\leq r< \infty$.
 With the boundary integral operator  in hand, we  then convert  the IBVP (\ref{1.1}) into an equivalent  nonlinear  integral equation whose (local)
 well-posedness can be established by using
 contraction mapping  principle thanks to the Strichartz estimate  (\ref{1.1-2})  of the boundary integral operator $W_b$.
The unconditional result is obtained in a similar way as that in \cite{Sun_CUC_Evo}. Here, again we emphasize that the same idea and method can be applied for (\ref{1.1}) with $ x\in \RR^n$.

The paper is organized as follows. Section 2 gives the formulation of the problem.
Section 3 deals with the representations of the solution operators and various Strichartz
estimates for these operators.
 The local well-posedness of the IBVP is proved in Section 4.
The global well-posedness of (\ref{1.1}) is provided in Section 5.

\section{Formulation of the problem}

To consider the well-posedness of (\ref{1.1}), i.e.,
\begin{equation}
\left\{
\begin{array}{lrl}
i u_t + u_{xx} + u_{yy} + g = 0\, , \quad & \qquad (x,y,t) & \in \R \times \R^+ \times (0,T) ; \\
u(x,y,0) = \varphi(x,y)\, , \quad & \qquad (x,y) & \in \R \times \R^+ \, ;\\
u(x,0,t) = h(x,t)\, , \quad & \qquad (x,t) & \in \R \times (0,T)\, ,
\end{array}
        \right. \label{2.1}
\end{equation}
where $g(x,y,t) := \lambda |u(x,y,t)|^{p-2} u(x,y,t)$ with  $ p \ge 3$,  we will reformulate it into an integral equation.
If $T> 0 $ is small, the well-posedness of (\ref{2.1}) is local and the solution $u$ in $C \left([0,T]; H^s (\R \times \R^+) \right)$ will be discussed.

First, we decompose (\ref{2.1}) into three relatively simple problems:
\[ u(x,y,t)= \tilde \mu (x,y,t)+ w(x,y,t)+ z(x,y,t)\, ,\]
where  $\tilde \mu $ solves
\begin{equation}
\left\{
\begin{array}{lrl}
i \tilde \mu_t + \tilde \mu_{xx} + \tilde \mu_{yy}  = 0\, , \quad & \qquad (x,y,t) & \in \R \times \R^+ \times (0,T) \, , \\
\tilde \mu(x,y,0) =0\, , \quad & \qquad (x,y) & \in \R \times \R^+ \, ,\\
\tilde \mu(x,0,t) = h(x,t)\, , \quad & \qquad (x,t) & \in \R \times (0,T),
\end{array}
        \right. \label{2.1-1}
\end{equation}
$w$ solves
\begin{equation}
\left\{
\begin{array}{lrl}
i w_t + w_{xx} + w_{yy}  = 0\, , \quad & \qquad (x,y,t) & \in \R \times \R^+ \times (0,T) \, , \\
w(x,y,0) = \varphi(x,y)\, , \quad & \qquad (x,y) & \in \R \times \R^+ \, ,\\
w(x,0,t) = 0\, , \quad & \qquad (x,t) & \in \R \times (0,T),
\end{array}
        \right. \label{2.1-2}
\end{equation}
and
$z$ solves
\begin{equation}
\left\{
\begin{array}{lrl}
i z_t + z_{xx} + z_{yy} + g = 0\, , \quad & \qquad (x,y,t) & \in \R \times \R^+ \times (0,T) \, , \\
z(x,y,0) = 0\, , \quad & \qquad (x,y) & \in \R \times \R^+ \, ,\\
z(x,0,t) = 0\, , \quad & \qquad (x,t) & \in \R \times (0,T)\, .
\end{array}
        \right. \label{2.1-3}
\end{equation}
The solution $\tilde \mu $ of the IBVP (\ref{2.1-1})  is denoted by
\[ \tilde \mu  (x,y,t):=W_b [h](x,y,t),\]
where the operator  $W_b$ is called the {\em boundary integral operator}.
Let $\phi: \, \R^2 \to \C$ be an extension of $\varphi$ such that
$\phi(x,y)=\varphi(x,y)$ for $y\ge0$ preserving the function spaces used later and the relative norms.
Let the solution of the linear initial value problem
\begin{equation}
            \left\{
            \begin{array}{lrl}
                i v_t + v_{xx} + v_{yy} = 0 \, ,\quad & \qquad (x,y,t) & \in \R^2 \times (0,T)\, ; \\
                v(x,y,0) = \phi(x,y) \, ,\quad & \qquad (x,y) & \in \R^2\, ,
            \end{array}
            \right. \label{2.2}
\end{equation}
be $v = W_{\R^2} (t) \phi$. Here, $W_{\R^2}(t)$ is a $C_0$-semigroup for the infinitesimal generator $i\Delta$ in $\R^2$. If $v_b $ is the trace of $v=W_{\R^2}(t) \phi $ at $y=0$,
 the solution $w$ of the IBVP (\ref{2.1-2}) can be expressed as
\[  w= W_{\R^2} (t) \phi  - W_b [v_b]   .\]
Also, if $f: \, \R^2 \times [0,T] \to \C$ is an extension of $g$  such that $f(x,y,t) = g(x,y,t)$ for $y\ge0$ preserving the function spaces and the relative norms, the  solution of the initial value problem
        \begin{equation}
            \left\{
            \begin{array}{lrl}
                i q_t + q_{xx} + q_{yy} + f = 0 \, ,\quad & \qquad (x,y,t) & \in \R^2 \times (0,T) \, ; \\
                q(x,y,0) = 0 \, ,\quad & \qquad (x,y) & \in \R^2\, ,
            \end{array}
            \right. \label{2.3}
        \end{equation}
can be found by $\displaystyle q= \Phi_f := i \int_0^t W_{\R^2}(t-\tau) f(\tau) \, d\tau$ as that in \cite{cazenave}.  If
$ q_b $ is the trace of $   i \int_0^t W_{\R^2}(t-\tau) f(\tau) \, d\tau$ at $  y=0$,
the solution $z$ of the IBVP (\ref{2.1-3}) can be expressed as
\[  z=i \int_0^t W_{\R^2}(t-\tau) f(\tau) \, d\tau  - W_b[ q_b ]  .\]
Consequently,  the IBVP  (\ref{2.1}) is  then transformed to  the following integral
equation,
    \begin{eqnarray}
       u(x,y,t) &=& W_b \left[h - W_{\R^2}(\cdot) \phi\, \Big |_{y=0} - i \left(\int_0^{\cdot} W_{\R^2}(\cdot-\tau) f(\tau) \, d\tau \right)\bigg | _{y=0} \right](x,y,t) \nonumber \\
       && \qquad + W_{\R^2}(t) \phi (x,y) + \, i \left(\int_0^t W_{\R^2}(t-\tau) f(\tau) \, d\tau \right) (x,y)\, . \label{2.6}
    \end{eqnarray}
From now on, we study the solution of (\ref{2.6}). In general, the solution of (\ref{2.6}) is a solution of (\ref{2.1}) in the sense of distribution.
If the initial and boundary data are smooth enough with compatibility conditions (here, the compatibility conditions are already in the operator $W_b$ if the part in \eqref{2.6} with $W_b$ is smooth at $t=y=0$) and the solutions of (\ref{2.6}) are in $C([0, T]; H^2 (\mathbb{R}\times \mathbb{R}^+)) \cap
C^1([0, T]; L^2 (\mathbb{R}\times \mathbb{R}^+))$, then it is straightforward to check that such solutions of (\ref{2.6}) are strong solutions of
(\ref{2.1}). For general initial and boundary data, we will only consider the solutions of (\ref{2.6}), which is consistent with the mild solutions of (\ref{2.1}) defined in Definition \ref{defi1.2}.

\medskip
One of the advantages to study the integral equation (\ref{2.6}) is that one only needs  to study the boundary integral operator $W_b$ since the operator $W_{\R^2}$
has been well-studied in literature.  However,   a necessary step to make this approach work is to find an appropriate function space ${\cal G}^s (0,T) $ such that
\begin{itemize}
\item[(i).] for any $h\in {\cal G}^s ( 0,T)$, $W_b[h] \in C([0,T]; H^s (\R\times \R^+)) $ and
\begin{equation}\label{x-1} \| W_b [h]\|_{C([0,T]; H^s (\R\times \R^+))}\leq C\|h\|_{{\cal G}^s (0,T)}; \end{equation}
\item[(ii).] for any $\phi \in H^s(\R^2)$, the trace $v_b$ of $v=W_{\R^2} (t) \phi $ at $y=0$ belongs to the space ${\cal G}^s (0,T)$ and
\begin{equation}\label{x-2} \|v_b\|_{{\cal G}^s (0,T)} \leq C\|\phi \|_{H^s \left (\R^2\right )} .\end{equation}
\end{itemize}

Based upon scaling argument  and the study of the IBVP of nonlinear Schr\"odinger equation on a half line \cite{holmer, Audiard-1},  it seems  reasonable to  choose
$ {\cal G}^s  (0,T)={\cal W}^s (0,T) $ defined in \eqref{1.2.1}.
Indeed,  it has been shown  in \cite{Audiard-1} that (\ref{x-1}) holds  for $s=1$  with
$ {\cal G}^s  (0,T)= {\cal W}^s (0,T).$
However, {\it does the trace of $W_{\R^2} \phi $ belong to the space ${\cal W}^s (0,T)?$} Next lemma shows that the trace
of $W_{\R^2}(t)\phi $ belongs to the space ${\cal H}^s (0,T)$.
\begin{lemma} \label{lemm2.2}   Let $s\geq 0$ be given.  For any $\phi \in H^s(\R^2)$,  if $v_b:=\left. W_{\R^2} (t)\phi \right |_{y=0}$, then
 there exists a constant $C>0$ such  that
$\| v_b\|_{{\cal H}^s(\R^2)} \leq C\|\phi \|_{H^s (\R^2)} $.
\end{lemma}
\begin{proof}
Since $u(x,y,t)=W_{\R^2} (t) \phi$ can be expressed as
\[ u(x,y,t)=\int _{\R} \int _{\R} e^{i (\xi ^2 +\eta ^2)t} e^{ix\xi+iy\eta} \hat{\phi} (\xi, \eta ) d\eta d\xi, \]
thus, the Fourier transform of $v_b (x,t)$ respect to $x$ is
\begin{align*}
\hat {v}_b (\xi ,t)  = &  \int _{\R} e^{i (\xi ^2 +\eta ^2)t }\hat{\phi } (\xi , \eta)d\eta = \int ^\infty_{\xi^2} e^{i w t } \frac{\hat{\phi } \left (\xi , \sqrt{w -\xi^2} \right ) }{  2\sqrt{w -\xi^2}}d w
 + \int ^\infty_{\xi^2} e^{i w t }\frac{\hat{\phi } \left (\xi , - \sqrt{w -\xi^2} \right )}{ 2\sqrt{w -\xi^2}} d w \\
= & \int ^\infty_{-\infty} e^{-i \lambda t }\zeta_{[\xi^2, \infty)}(-\lambda ) \frac{\hat{\phi } \left (\xi , \sqrt{|\lambda +\xi^2|} \right )}{ 2\sqrt{|\lambda +\xi^2|}} d\lambda+ \int ^\infty_{-\infty} e^{- i \lambda t }\zeta_{[\xi^2, \infty)}(-\lambda)\frac{\hat{\phi } \left (\xi , - \sqrt{|\lambda +\xi^2|} \right ) }{ 2\sqrt{|\lambda +\xi^2|}} d\lambda
\end{align*}
satisfying
\begin{align*}
\|v_b \|_{{\cal H}^s(\R^2)}^2  \leq & \int _{\R}\int_{\R} (1+|\lambda|+\xi ^2) ^s
\zeta_{[\xi^2, \infty)}(-\lambda ) \frac{\left | \hat{\phi } \left (\xi , \sqrt{|\lambda +\xi^2|} \right )\right |^2}{ 4\sqrt{|\lambda +\xi^2|}} d\lambda d\xi\\
& + \int _{\R}\int_{\R} (1+|\lambda|+\xi ^2) ^s
\zeta_{[\xi^2, \infty)}(-\lambda ) \frac{\left | \hat{\phi } \left (\xi , - \sqrt{|\lambda +\xi^2|} \right )\right |^2}{ 4\sqrt{|\lambda +\xi^2|}} d\lambda d\xi\\
\leq &
\int _{\R}\int_{\xi^2}^\infty (1+|\lambda|+\xi ^2) ^s
 \frac{\left | \hat{\phi } \left (\xi , \sqrt{|-\lambda +\xi^2|} \right )\right |^2}{ 4\sqrt{|-\lambda +\xi^2|}} d\lambda d\xi\\
& + \int _{\R}\int_{\xi^2}^\infty (1+|\lambda|+\xi ^2) ^s
\frac{\left | \hat{\phi } \left (\xi , - \sqrt{|-\lambda +\xi^2|} \right )\right |^2}{ 4\sqrt{|-\lambda +\xi^2|}} d\lambda d\xi\\
\leq & C \int _{\R}\int_{\R }   (1+|\xi^2 +\eta^2 |+\xi ^2) ^s  {|\hat{\phi} (\xi, \eta )|^2} d\eta d \xi
 \leq   C  \int _{\R}\int_{\R} (1+\eta ^2+\xi ^2) ^s |\hat{\phi }(\xi, \eta )|^2 d\eta d\xi\, ,
\end{align*}
where $\zeta$ is the characteristic function.
\end{proof}
Proposition \ref{prop3.2}  in  Section  3 shows that (\ref{x-1}) holds with ${\cal G}^s(0,T)={\cal H}^s (0,T)$. Thus the space ${\cal H}^s (0,T)$ is indeed the optimal choice
of $h$.
Obviously, the space ${\cal W}^s (0,T)$ is a subspace of ${\cal H}^s (0,T)$.


\section{Representations and estimates of boundary integral operators}

In this section, we will derive the formulas of the solutions in (\ref{2.1-1})-(\ref{2.1-3}) and obtain the corresponding solution operators and their relevant estimates.
We begin with considering the following nonhomogeneous boundary value problem:
\begin{equation}
        \left\{
        \begin{array}{lrl}
            i u_t + u_{xx} + u_{yy} = 0\, , \quad & \qquad (x,y,t) & \in \R \times \R^+ \times (0,T)\, ; \\
            u(x,y,0) = 0\, , \quad & \qquad (x,y) & \in \R \times \R^+ \, ,\\
            u(x,0,t) = h(x,t) \quad & \qquad (x,t) & \in \R \times (0,T)\, .
        \end{array}
        \right. \label{3.1}
    \end{equation}
We will show that if the boundary condition  $h (x, t) $  is in the space ${\cal H}^s (0,T)$,  the solution $W_b [h]$ is  in $ L^q_t\left([0,T]; \, W^{s,r}_{xy}(\R \times \R^+)\right) \bigcap L^{\infty}_t \left([0,T]; \, H^s_{xy}(\R \times \R^+)\right)$, where $(q,r)$ is an admissible pair, i.e., $\frac{1}{q} + \frac{1}{r} = \frac{1}{2}$.
In the sequel, we use $ \lesssim $ or $\eqsim $ to denote $\leq $ or $ = $ with a difference of a factor of a generic constant.

First, we derive an  explicit formula of solution operator $W_b [h] $.
\begin{prop}\label{prop3.1} 
The solution of (\ref{3.1}) can be expressed by
\begin{eqnarray}
            u(x,y,t) = W_b [h] (x,y,t) &=& \dfrac{1}{\pi^2}\left(\int_{-\infty}^{\infty} e^{i\xi x}\left(\int_0^{\infty} e^{-i(\xi^2-\eta^2)t-y\eta}\eta\widetilde{h}(\xi, -i(\xi^2-\eta^2)) \, d\eta \right.\right. \nonumber \\
            && \left.\left. + \int_0^{\infty} e^{-i(\xi^2+\eta^2)t+iy\eta}\eta\widetilde{h}(\xi, -i(\xi^2+\eta^2)) \, d\eta \right) \, d\xi \right)  \nonumber \\
            &\eqsim& W_{b_1}[h](x,y,t) + W_{b_2}[h](x,y,t) \, , \label{3.2}
\end{eqnarray}
where
\begin{eqnarray}
W_{b_1}[h](x,y,t) &=& \int_{-\infty}^{\infty} \int_0^{\infty} e^{-i(\xi^2+\eta^2)t+i(y\eta+x\xi)}\eta\widetilde{h}(\xi, -i(\xi^2+\eta^2)) \, d\eta \, d\xi \, ,\label{3.3} \\
W_{b_2}[h](x,y,t) &=& \int_{-\infty}^{\infty} \int_0^{\infty} e^{-i(\xi^2-\eta^2)t+ix\xi-y\eta}\eta\widetilde{h}(\xi, -i(\xi^2-\eta^2)) \, d\eta \, d\xi \, . \label{3.4}
\end{eqnarray}
\end{prop}
\begin{proof}
For any fixed $t > 0 $, we apply Fourier transform on both sides of the equation in (\ref{3.1}) with respect to $x\in\R$ so that the equation of $u(x,y,t)$ is converted to the following equation of $\widehat{u^x}(\xi,y,t)$ (here, $\widehat{u^x}(\xi,y,t)$ is the Fourier transform of $u$ with respect to $x$),
\begin{equation}
            \left\{
            \begin{array}{lrl}
                i \widehat{u^x}_t - \xi^2\widehat{u^x} + \widehat{u^x}_{yy} = 0\, , & \qquad (\xi,y,t) & \in \R \times \R^+ \times (0,\infty) \, ;\\
                \widehat{u^x}(\xi,y,0) = 0 \, ,\\
                \widehat{u^x}(\xi,0,t) = \widehat{h^x}(\xi,t)\, .
            \end{array}
            \right. \label{3.5}
\end{equation}
Then, apply Laplace transform on both sides of (\ref{3.5}) with respect to $t>0$ (here $\widehat{u^x}(\xi,y,0) = 0$) with $\widetilde{u}(\xi,y,\omega) = \mathcal{L}_t[\widehat{u^x}](\xi,y,\omega)$ and $\widetilde{h}(\xi,\omega) = \mathcal{L}_t[\widehat{h^x}](\xi,\omega)$, which gives
\begin{equation} 
            \left\{
            \begin{array}{lrl}
                (i\omega-\xi^2)\widetilde{u} + \widetilde{u}_{yy} = 0\, , & \qquad (\xi,y,\omega) & \in \R \times \R^+ \times \C\, ; \\
                \widetilde{u}(\xi,0,\omega) = \widetilde{h}(\xi,\omega)\, .
            \end{array}
            \right. \label{3.6}
\end{equation}
The characteristic equation of (\ref{3.6}) is $z^2+\omega i-\xi^2=0 $ or $z^2 = \xi^2-\eta-i\gamma $, where $\omega=\gamma-\eta i$ with $\gamma>0$ is required in the form of inverse Laplace transform. If $z=re^{i\theta}$ with $\theta=\arg z$, then
$r^2\sin(2\theta)=-\gamma<0$, which implies $ \frac{\pi}{2}< \theta< \pi$ or $\frac{3\pi}{2} < \theta < 2\pi$. Therefore,
 two roots of the characteristic equation are $z_1=re^{i\theta}$ and $z_2=re^{i(\theta+\pi)}$ with $\frac{\pi}{2}<\theta<\pi$.
Since $\tilde{u}\rightarrow 0$ as $y\rightarrow \infty$, $\tilde{u}=c(\xi, \omega)e^{re^{i\theta}y}=c(\xi, \omega)e^{ry(\cos\theta+i\sin\theta)}$ where
$$
r=\sqrt[4]{(\xi^2-\eta)^2+\gamma^2},   \quad \cos2\theta=\frac{\xi^2-\eta}{\sqrt{(\xi^2-\eta)^2+\gamma^2}}\, ,\quad  \sin2\theta=\frac{-\gamma}{\sqrt{(\xi^2-\eta)^2+\gamma^2}}.
$$
If $\gamma\rightarrow0^+$, then $r\rightarrow \sqrt{|\xi^2-\eta|}$. We divide it into two cases.
For $\eta<\xi^2$, $\cos2\theta\rightarrow1$ and $\sin2\theta\rightarrow0^-$ with $\frac{\pi}{2}<\theta<\pi$, which yields $ \cos\theta\rightarrow-1 \, ,  \sin\theta\rightarrow0^+ $ or $ z\rightarrow -\sqrt{\xi^2-\eta}$. Thus,  $\tilde{u}(\xi,y,\omega)=\widetilde{h}(\xi, \omega)e^{-y\sqrt{\xi^2-\eta}}$.
For $\eta>\xi^2$, $\cos2\theta\rightarrow-1$ and $\sin2\theta\rightarrow0^-$, which implies that  $\cos\theta\rightarrow0^- \, , \, \sin\theta\rightarrow1 $ or $ z\rightarrow i\sqrt{\eta-\xi^2}$. Thus,  $\tilde{u}(\xi,y,\omega)=\widetilde{h}(\xi, \omega)e^{iy\sqrt{\eta-\xi^2}}$.

Perform the inverse Laplace transform on $\widetilde{u^x}$ with respect to $\omega$.
If $\gamma\rightarrow 0^+$ with $\omega=\gamma-i\eta$, then
\begin{align*}
            \widehat {u^x}(\xi, y, t) &= \dfrac{1}{2\pi i} \int_{\gamma-i\infty}^{\gamma+i\infty} e^{\omega t}e^{zy}\widetilde{h}(\xi, \omega) \, d\omega \nonumber\\
            &\overset{\gamma\rightarrow 0}{=} \, \, \dfrac{-1}{2\pi}\left(\int_{\infty}^{\xi^2} e^{-i\eta t}e^{iy\sqrt{\eta-\xi^2}}\widetilde{h}(\xi, -i\eta) \, d\eta + \int_{\xi^2}^{-\infty} e^{-i\eta t}e^{-y\sqrt{\xi^2-\eta}}\widetilde{h}(\xi, -i\eta) \, d\eta \right) \nonumber \\
            &= \dfrac{1}{\pi}\left(\int_0^{\infty} e^{-i(\xi^2+\eta^2)t+iy\eta}\eta\widetilde{h}\big (\xi, -i(\xi^2+\eta^2)\big ) \, d\eta + \int_0^{\infty} e^{-i(\xi^2-\eta^2)t-y\eta}\eta\widetilde{h}\big (\xi, -i(\xi^2-\eta^2)\big ) \, d\eta \right)\, ,
\end{align*}
where we have replaced $\eta$ by $(\xi^2+\eta^2)$ in the first integral and $\eta$ by $(\xi^2-\eta^2)$ in the second integral.
Finally, we can take the inverse Fourier transform on $\xi$ to find
        \begin{eqnarray*}
            u(x,y,t) &=& \dfrac{1}{\pi^2}\left(\int_{-\infty}^{\infty} e^{i\xi x} \left(\int_0^{\infty} e^{-i(\xi^2-\eta^2)t-y\eta}\eta\widetilde{h}\big (\xi, -i(\xi^2-\eta^2)\big ) \, d\eta \right.\right. \\
            && \left.\left. + \int_0^{\infty} e^{-i(\xi^2+\eta^2)t+iy\eta}\eta\widetilde{h}\big (\xi, -i(\xi^2+\eta^2)\big ) \, d\eta \right) \, d\xi \right) \\
            &=& W_{b_1}[h](x,y,t) + W_{b_2}[h](x,y,t)\, ,
        \end{eqnarray*}
or $W_b[h](x,y,t) = W_{b_1}[h](x,y,t) + W_{b_2}[h](x,y,t)$ as the solution of (\ref{3.1}).
\end{proof}

To study $W_{b_1}[h](x,y,t)$ and $W_{b_2}[h](x,y,t)$,  rewrite them to
more convenient forms. For $W_{b_1}[h](x,y,t)$, we let $\boldsymbol{x} = (x,y)$ and $\boldsymbol{\xi} = (\xi,\eta) \in \R^2$ and define
\begin{equation*}
        \widehat{\Phi_h}(\boldsymbol{\xi}) = \nu_1(\xi,\eta) =
        \left\{
        \begin{array}{lcl}
            \eta \widetilde{h}\big (\xi, -i(\xi^2+\eta^2)\big ) &, & \eta\ge 0 \, ,\\
            0 &, & \eta<0\, .
        \end{array}
        \right.
\end{equation*}
Then,
\begin{equation}
        W_{b_1}[h](x,y,t) = \int_{-\infty}^{\infty} \int_0^{\infty} e^{-i(\xi^2+\eta^2)t+i(y\eta+x\xi)} \widehat{\Phi_h}(\xi,\eta) \, d\eta \, d\xi = \int_{\R^2} e^{-i|\boldsymbol{\xi}|^2t + i \boldsymbol{\xi} \cdot \boldsymbol{x}}\widehat{\Phi_h}(\boldsymbol{\xi}) \, d \boldsymbol{\xi} \, ,\label{3.7}
\end{equation}
which is exactly the integral solution formula of initial value problem (\ref{2.2}) for the linear Schr\"odinger equation over the whole plane $\R^2$, i.e., $W_{b_1}[h](x,y,t) = W_{\R^2}(t)\Phi_h (x,y)$.
Similarly, if
    \begin{equation*}
        \widehat{\Psi_h}(\xi,\eta) = \nu_2(\xi,\eta) =
        \left\{
        \begin{array}{lcl}
            \eta \widetilde{h}\big (\xi, -i(\xi^2-\eta^2)\big ) &, & \eta\ge 0\, , \\
            0 &, & \eta<0\, ,
        \end{array}
        \right.
    \end{equation*}
then
\begin{equation}
    W_{b_2}[h](x,y,t) = \int_{\R^2} e^{-i(\xi^2-\eta^2)t + ix\xi-y\eta} \widehat{\Psi_h}(\xi,\eta) \, d\eta \, d\xi\, .
    \label{3.8}
\end{equation}
Note that $\widehat{\Psi_h}(\xi,\eta)$ is continuous in $\eta$ at $\eta=0$.

Next, we derive the estimates for the operator $W_b [h]$.
\begin{prop}\label{prop3.2}
For $\frac{1}{q}+\frac{1}{r} = \frac{1}{2}$ with $r\in [2, \infty)$, the following estimates hold,
\begin{align}
\|W_b [h]  & \|_{L^q_t([0,T]; \, W^{s,r}_{xy} (\R \times \R^+))}^2  \nonumber \\
 & \lesssim\int_{-\infty}^{\infty} \int_{-\infty}^{\infty} (1+|\beta|+\xi^2)^s \sqrt{|\beta|} \left(\int_0^{\infty} e^{i(\beta+\xi^2)t} \int_{-\infty}^{\infty} e^{-i\xi x} h(x,t) \, dx \,
 dt \right)^2 \,d\beta \, d\xi \, \nonumber \\ &  \lesssim  \|h\|_{{\cal H}^s(\R^2)}^2 \, .\label{3.9}
 \end{align}
\end{prop}
\begin{proof}
We need derive the estimates for both $W_{b_1}[h]$ and $W_{b_2}[h]$. We first prove (\ref{3.9}) for $W_{b_1}[h]$.
From the form of $W_{b_1}[h]$ in (\ref{3.3}) and the definition of $\Phi_h$ in (\ref{3.7}), we have
\begin{align}
                    \|\Phi_h\|^2_{H^s} &= \left \|\left(1+\xi^2+\eta^2\right)^{\frac{s}{2}}\widehat{\Phi_h}(\xi,\eta)\right \|^2_{L^2} \nonumber \\
                    &= \int_{-\infty}^{\infty} \int_0^{\infty} \left(1+\xi^2+\eta^2\right)^s \eta^2 \left|\widetilde{h}\big (\xi,-i(\xi^2+\eta^2)\big )\right|^2 \, d\eta \, d\xi \nonumber\\
                    &= \int_{-\infty}^{\infty} \int_0^{\infty} \left(1+\xi^2+\eta^2\right)^s \eta^2 \left|\int_0^{\infty} e^{i(\xi^2+\eta^2)t}\widehat{h^x}(\xi,t) \, dt\right|^2 \, d\eta \, d\xi \nonumber\\
                    &\eqsim \int_{-\infty}^{\infty} \int_0^{\infty} \left(1+\xi^2+\eta^2\right)^s \eta^2 \left|\int_0^{\infty} e^{i(\xi^2+\eta^2)t} \int_{-\infty}^{\infty} e^{-i\xi x} h(x,t) \, dx \, dt\right|^2 \, d\eta \, d\xi \nonumber\\
                    &= \frac{1}{2} \int_{-\infty}^{\infty} \int_0^{\infty} (1+\beta+\xi^2)^s \dfrac{\beta}{\sqrt{\beta}} \left|\int_0^{\infty} e^{i(\beta+\xi^2)t} \int_{-\infty}^{\infty} e^{-i\xi x} h(x,t) \, dx \, dt \right|^2 \, d\beta \, d\xi \nonumber\\
                    &= C \int_{-\infty}^{\infty} \int_0^{\infty} (1+\beta+\xi^2)^s \sqrt{\beta} \left|\int_0^{\infty} \int_{-\infty}^{\infty} e^{i\beta t} e^{i\xi^2t} e^{-i\xi x} h(x,t) \, dx \, dt \right|^2 \, d\beta \, d\xi\label{3.11.5}
\end{align}
(replacing $\eta$ by $\sqrt{\beta}$), and
\begin{equation*}
                    \|W_{b_1}[h]\|_{L^q_t\left(\R^+; \, W^{s,r}_{xy}(\R \times \R^+)\right)} \lesssim \|\Phi_h\|_{H^s_{xy}(\R \times \R^+)}
\end{equation*}
for any admissible pair $(q,r)$ with $r\in[2,\infty)$, by
the fact that $W_{b_1}[h] = W_{\R^2}(t) \Phi_h$ solves (\ref{2.2}) and the Strichartz estimate from \emph{Remark 2.3.8} in \cite{cazenave}.
By combining these two estimates, it is deduced that
\begin{align}
                     \|W_{b_1}[h] & \|_{L^q_t \left([0,T]; \, W^{s,r}_{xy}(\R \times \R^+)\right)} \nonumber \\
                    &\lesssim \left\{\int_{-\infty}^{\infty} \int_0^{\infty} (1+|\beta|+\xi^2)^s \sqrt{|\beta|} \left(\int_0^{\infty} \int_{-\infty}^{\infty} e^{i\left(\beta+\xi^2\right)t} e^{-i\xi x} h(x,t) \, dx \, dt \right)^2 \, d\beta \, d\xi \right\}^{\frac{1}{2}}\, .\label{3.11}
\end{align}

For $W_{b_2}[h]$, we first consider the case $s=0$, i.e., in the space $L^2$.
Several steps are needed to prove that
\begin{equation}
                    \|W_{b_2}[h]\|_{L^q_t \left([0,T]; \, L^r_{xy}(\R \times \R^+)\right)} \lesssim \|\Psi_h\|_{L^2(\R \times \R^+)}\, .\label{3.12}
\end{equation}
We rewrite $W_{b_2}[h]$ from (\ref{3.8}) by
\begin{align*}
 W_{b_2}[h](x,y,t)
                    &= \int_{-\infty}^{\infty} \int_0^{\infty} e^{-i(\xi^2-\eta^2)t+ix\xi-y\eta} \widehat{\Psi}_h(\xi, \eta) \, d\eta \, d\xi = \int_{-\infty}^{\infty} \int_{-\infty}^{\infty} e^{-i(\xi^2-\eta^2)t+ix\xi-y|\eta|} \widehat{\Psi}_h(\xi, \eta) \, d\eta \, d\xi \\
                    &= \int_{-\infty}^{\infty} \int_{-\infty}^{\infty} e^{-i(\xi^2-\eta^2)t+ix\xi-y|\eta|} \left[\int_{-\infty}^{\infty} \int_{-\infty}^{\infty} \Psi_h(\rho, \tau) \cdot e^{-i(\xi\rho+\eta\tau)} \, d\rho \, d\tau \right] \, d\eta \, d\xi \\
                    &= \int_{-\infty}^{\infty} \int_{-\infty}^{\infty} \Psi_h(\rho, \tau) \cdot \left[\int_{-\infty}^{\infty} \int_{-\infty}^{\infty} e^{-i(\xi^2-\eta^2)t+ix\xi-y|\eta|-i(\xi\rho+\eta\tau)} \, d\eta \, d\xi \right] \, d\rho \, d\tau\, .
\end{align*}
After denoting
\begin{align}
                     K_t (x,y,\rho,\tau) & = \int_{-\infty}^{\infty} \int_{-\infty}^{\infty} e^{-i(\xi^2-\eta^2)t+ix\xi-y|\eta|-i(\xi\rho+\eta\tau)} \, d\eta \, d\xi \nonumber \\
                    & = \left(\int_{-\infty}^{\infty} e^{-i\xi^2t+ix\xi-i\xi\rho} \, d\xi \right) \cdot \left(\int_{-\infty}^{\infty} e^{i\eta^2t-y|\eta|-i\eta\tau} \, d\eta \right)\, ,\label{3.13}
\end{align}
we can write $W_{b_2}[h]$ as
\begin{equation} 
                    W_{b_2}[h](x,y,t) := \mathcal{K}(t)\Psi_h(x,y) = \int_{-\infty}^{\infty} \int_{-\infty}^{\infty} \Psi_h(\rho, \tau) \cdot K_t (x,y,\rho,\tau) \, d\rho \, d\tau\, ,\label{3.14}
\end{equation}
where $\mathcal{K}(t)$ is an operator acting on each reasonable choice of $\Psi_h \in L^2(\R \times \R^+)$ with an integral kernel $K_t$.

The next task is to prove that
\begin{equation} 
                    \int_0^T \left <\mathcal{K}(t)\psi, \phi(t)\right >_{L^2_{xy} (\R \times \R^+)} \, dt \lesssim \|\psi\|_{L^2(\R \times \R^+)} \cdot \|\phi\|_{L^{q'}_t \left([0,T]; \, L^{r'}_{xy}(\R \times \R^+)\right)}\label{3.15}
\end{equation}
for $\phi \in L_t^{q'}\left([0,T]; \, L_{xy}^{r'}(\R \times \R^+)\right)$ and $\psi \in L^2_{xy}$, where $q'$ and $r'$ are conjugate indices of $q$ and $r$, respectively.

The reason of making efforts to (\ref{3.15}) is that its left hand side $\displaystyle \int_0^T \left<\mathcal{K}(t)\psi, \phi(t)\right>_{L^2_{xy} (\R \times \R^+)} \, dt$, which is given by the inner product of $\mathcal{K} (\cdot)\psi$ and $\phi$ over $L^2_{xyt} (\R \times \R^+ \times [0,T])$, corresponds to a linear functional on $L_t^{q'}\left([0,T]; \, L_{xy}^{r'}(\R \times \R^+)\right)$. Since $\|\mathcal{K} (\cdot)\psi\|_{L_t^{q'}\left([0,T]; \, L_{xy}^{r'}(\R \times \R^+)\right)}$ is the norm of this functional, if (\ref{3.15}) can be proven, then (\ref{3.12}) is proved according to the definition of functional norm.
Thus, it is necessary to define another operator $\mathcal{K}_T$ that maps each $\phi \in L_t^{q'}\left([0,T]; \, L_{xy}^{r'}(\R \times \R^+)\right)$ to a function without $t$, that is,
\begin{equation} 
                    \mathcal{K}_T\phi (\rho,\tau) := \int_0^T \int_{-\infty}^{\infty} \int_0^{\infty} K_t (x,y,\rho,\tau) \overline{\phi(x,y,t)} \, dy \, dx \, dt\, .\label{3.16}
\end{equation}
Moreover, we have
\begin{align*}
                    \int_0^T & \left < \mathcal{K}(t)\psi,  \phi(\cdot, \cdot, t) \right >_{L^2_{xy} (\R \times \R^+)} \, dt
                    = \int_0^T \int_{-\infty}^{\infty} \int_0^{\infty} \mathcal{K}(t)\psi \, (x,y) \cdot \overline{\phi(x,y,t)} \, dx \, dy \, dt \\
                    &= \int_0^T \int_{-\infty}^{\infty} \int_0^{\infty} \int_{-\infty}^{\infty} \int_{-\infty}^{\infty} \psi(\rho, \tau) \cdot K_t (x,y,\rho,\tau) \, d\rho \, d\tau \cdot \overline{\phi(x,y,t)} \, dx \, dy \, dt \\
                    &= \int_{-\infty}^{\infty} \int_{-\infty}^{\infty} \psi(\rho, \tau) \, \underbrace{\int_0^T \int_{-\infty}^{\infty} \int_0^{\infty} K_t (x,y,\rho,\tau) \overline{\phi(x,y,t)} \, dy \, dx \, dt}_{\mathcal{K}_T\phi (\rho,\tau)} \, d\rho \, d\tau\\
                    &\le C \|\psi\|_{L^2_{\rho,\tau}(\R \times \R^+)} \cdot \|\mathcal{K}_T\phi\|_{L^2_{\rho\tau}(\R \times \R^+)}\, ,
\end{align*}
which yields
\begin{equation} 
                    \int_0^T \left<\mathcal{K}(t)\psi, \phi(t)\right>_{L^2_{xy} (\R \times \R^+)} \, dt \lesssim \|\psi\|_{L^2_{\rho\tau}(\R \times \R^+)} \cdot \|\mathcal{K}_T\phi\|_{L^2_{\rho\tau}(\R \times \R^+)}\, .\label{3.17}
\end{equation}

Now we need one more step from (\ref{3.17}) to (\ref{3.15}) and then can achieve the expected result. The following arguments mainly focus on the connection between $\|\phi\|_{L^{q'}_t \left([0,T]; \, L^{r'}_{xy}(\R \times \R^+)\right)}$ and $\|\mathcal{K}_T\phi\|_{L^2_{\rho\tau}(\R \times \R^+)}$.
Recall $K_t$ in (\ref{3.13}) and let
                \begin{equation} 
                        K_{s,\sigma} (x,y,z,w) := \int_{-\infty}^{\infty} \int_{-\infty}^{\infty} K_s (x,y,\rho,\tau) \overline{K_{\sigma} (z,w,\rho,\tau)} \, d\rho \, d\tau\, .\label{3.18}
                \end{equation}
Then, it is obtained that
\begin{align*}
\|\mathcal{K}_T\phi\|^2_{L^2}
=& \int_{-\infty}^{\infty} \int_{-\infty}^{\infty} \left(\int_0^T \int_{-\infty}^{\infty} \int_0^{\infty} K_s (x,y,\rho,\tau) \overline{\phi(x,y,s)} \, dy \, dx \, ds \right)  \\
& \qquad \times \left(\int_0^T \int_{-\infty}^{\infty} \int_0^{\infty} \overline{K_{\sigma} (z,w,\rho,\tau)} \phi(z,w,\sigma) \, dw \, dz \, d\sigma \right) \, d\rho \, d\tau \\
=& \int_0^T \int_0^T \bigg [ \int_{-\infty}^{\infty} \int_{-\infty}^{\infty} \int_0^{\infty} \int_{-\infty}^{\infty} \int_0^{\infty} \int_{-\infty}^{\infty} \overline{\phi(x,y,s)} \phi(z,w,\sigma) \\
& \qquad \times K_s (x,y,\rho,\tau) \overline{K_{\sigma} (z,w,\rho,\tau)} dx \, dy \, dz \, dw \, d\rho \, d\tau \bigg ] \, ds \, d\sigma \\
=& \int_0^T \int_0^T \bigg [\int_{-\infty}^{\infty} \int_{-\infty}^{\infty} \int_0^{\infty} \int_0^{\infty} \overline{\phi(x,y,s)} \phi(z,w,\sigma)  \\
& \qquad \times \left(\int_{-\infty}^{\infty} \int_{-\infty}^{\infty} K_s (x,y,\rho,\tau) \overline{K_{\sigma} (z,w,\rho,\tau)} \, d\rho \, d\tau \right) dw \, dy \, dz \, dx \bigg ] \, ds \, d\sigma \\
=& \int_0^T \int_0^T \left(\int_{-\infty}^{\infty} \int_{-\infty}^{\infty} \int_0^{\infty} \int_0^{\infty} \overline{\phi(x,y,s)} \phi(z,w,\sigma) K_{s,\sigma}(x,y,z,w) dw \, dy \, dz \, dx \right) \, ds \, d\sigma\, .
\end{align*}
For $y$, $w\ge0$,
\begin{align*}
                    K_{s,\sigma} &= \int_{-\infty}^{\infty} \int_{-\infty}^{\infty} K_s (x,y,\rho,\tau) \overline{K_{\sigma} (z,w,\rho,\tau)} \, d\rho \, d\tau \\
                    &= \int_{-\infty}^{\infty} \int_{-\infty}^{\infty} \left(\int_{-\infty}^{\infty} \int_{-\infty}^{\infty} e^{-i(\xi^2-\eta^2)s+ix\xi-y|\eta|-i(\xi\rho+\eta\tau)} \, d\eta \, d\xi \right) \\
                    & \qquad \times \left(\int_{-\infty}^{\infty} \int_{-\infty}^{\infty} e^{i(\tilde{\xi}^2-\tilde{\eta}^2)\sigma-iz\tilde{\xi}-w|\tilde{\eta}|+i(\tilde{\xi}\rho+\tilde{\eta}\tau)} \, d\tilde{\eta} \, d\tilde{\xi} \right) \, d\rho \, d\tau \\
                    &= (2\pi)^2 \int_{-\infty}^{\infty} \int_{-\infty}^{\infty} e^{-i(\xi^2-\eta^2)s+ix\xi-y|\eta|} \bigg [ \int_{-\infty}^{\infty} \int_{-\infty}^{\infty} e^{-i(\xi\rho+\eta\tau)}  \\
                    & \qquad \times \left(\dfrac{1}{(2\pi)^2} \int_{-\infty}^{\infty} \int_{-\infty}^{\infty} e^{i(\tilde{\xi}^2-\tilde{\eta}^2)\sigma-iz\tilde{\xi}-w|\tilde{\eta}|} e^{i(\tilde{\xi}\rho+\tilde{\eta}\tau)} \, d\tilde{\eta} \, d\tilde{\xi} \right) \, d\rho \, d\tau \bigg ] \, d\eta \, d\xi \\
                    &= (2\pi)^2 \int_{-\infty}^{\infty} \int_{-\infty}^{\infty} e^{i(-\xi^2-\eta^2)s+ix\xi-y|\eta|} \mathcal{F}_{\rho, \tau} \circ \mathcal{F}^{-1}_{\tilde{\xi}, \tilde{\eta}} \left[ e^{i(\tilde{\xi}^2-\tilde{\eta}^2)\sigma-iz\tilde{\xi}-w\tilde{\eta}} \right] \, d\eta \, d\xi \\
                    &= (2\pi)^2 \int_{-\infty}^{\infty} \int_{-\infty}^{\infty} e^{-i(\xi^2-\eta^2)(s-\sigma)+i\xi(x-z)-|\eta|(y+w)} \, d\eta \, d\xi \\
                    &= C \left(\int_{-\infty}^{\infty} e^{-i\xi^2(s-\sigma)+i\xi(x-z)} \, d\xi \right) \cdot \left(\int_{-\infty}^{\infty} e^{i\eta^2(s-\sigma)-|\eta|(y+w)} \, d\eta \right)\, .
\end{align*}
If we define
                \begin{equation}
                    \widetilde{K}_t (x,y,z,w) = \left(\int_{-\infty}^{\infty} e^{-i\xi^2t+i\xi(x-z)} \, d\xi \right) \cdot \left(\int_{-\infty}^{\infty} e^{i\eta^2t-|\eta|(y+w)} \, d\eta \right) \qquad \text{\ for $y>0$ \ and \ $w>0$}\, ,\label{3.19}
                \end{equation}
 then $K_{s,\sigma} (x,y,z,w) = \widetilde{K}_{s-\sigma} (x,y,z,w)$.
Now, we use $\widetilde{K}_{s-\sigma}$ to study the operator
\begin{equation}
                    \widetilde{\mathcal{K}}(t)\varphi (x,y) = \int_{-\infty}^{\infty} \int_0^{\infty} \varphi(z,w) \widetilde{K}_t (x,y,z,w) \, dz \, dw
                    \label{3.20}
\end{equation}
for  $\varphi \in C_0^{\infty} (\R \times \R^+)$.
To estimate the kernel $\widetilde{K}_t$, for $(x,y,z,w) \in \R \times \R^+ \times \R \times \R^+$, it is derived that
\begin{eqnarray*}
                    \left|\int_{-\infty}^{\infty} e^{-i\xi^2t+i(x-z)\xi} \, d\xi \right| &\overset{\xi\sqrt{t} \rightarrow \xi_0}{=}& \dfrac{1}{\sqrt{t}} \left| \int_{-\infty}^{\infty} e^{-i\xi_0^2+i\frac{(x-z)}{\sqrt{t}}\xi_0} \, d\xi_0 \right| \\
                    &\overset{\left(\xi_0-\frac{x-z}{2\sqrt{t}} \right) \rightarrow \xi_1}{=}& \dfrac{1}{\sqrt{t}} \left| \int_{-\infty}^{\infty} e^{-i\xi_1^2} \, d\xi \right| = \dfrac{C}{\sqrt{t}}\, .
\end{eqnarray*}
An estimate from \cite{Sun_NSE_1d} yields
                \begin{equation*}
                    \left|\int_0^{\infty} e^{i\eta^2t-a\eta-i\eta b} \, d\eta \right| \le \dfrac{C}{\sqrt{t}} \quad \text{for any } a \in \R^+ \, ,  b \in \R\, .
                \end{equation*}
By substituting $y+w$ for $a$ and $0$ for $b$, we obtain
                \begin{equation*}
                    \left|\int_{-\infty}^{\infty} e^{i\eta^2t-|\eta|(y+w)} \, d\eta \right| \le \dfrac{C}{\sqrt{t}}\, .
                \end{equation*}
Hence
\begin{eqnarray}
                    \left|\widetilde{K}_t (x,y,z,w) \right| = \left|\int_{-\infty}^{\infty} e^{-i\xi^2t+ix\xi-i\xi z} \, d\xi \right| \cdot \left|\int_{-\infty}^{\infty} e^{i\eta^2t-|\eta|(y+w)} \, d\eta \right|  \le \dfrac{C}{\sqrt{t}} \cdot \dfrac{C}{\sqrt{t}}
                    \le \dfrac{C}{t} \, .\label{3.21}
\end{eqnarray}
For  some fixed $t> 0 $,   (\ref{3.20}) and (\ref{3.21}) imply
                \begin{eqnarray}
                    \left\|\widetilde{\mathcal{K}}(t)\varphi \right\|_{L^{\infty}_{xy}} &\le& \int_{-\infty}^{\infty} \int_0^{\infty} \left|\varphi(z,w) \widetilde{K}_t (x,y,z,w)\right| \, dw \, dz \nonumber \\
                    &\le& \dfrac{C}{|t|} \int_{-\infty}^{\infty} \int_0^{\infty} |\varphi(z,w)| \, dw \, dz
                    = C \cdot |t|^{-1} \|\varphi\|_{L^1_{xy}} \label{3.22}
                \end{eqnarray}
for each fixed $t>0$.
On the other hand, for $t > 0 $, we can also prove that $\|\widetilde{\mathcal{K}}(t)\varphi\|_{L^2 (\R \times \R^+)} \le \|\varphi\|_{L^2 (\R \times \R^+)}$ by an estimate $ \left\|\int_0^{\infty} e^{-y\eta} f(\eta) \, d\eta \, \right\|_{L^2 (\R^+)} \le \|f\|_{L^2 (\R^+)}$. From
\begin{eqnarray*}
                    \widetilde{\mathcal{K}}(t) \varphi (x,y) &=& \int_{-\infty}^{\infty} \int_0^{\infty} \varphi(z,w) \widetilde{K}_t (x,y,z,w) \, dw \, dz  \\
                    &=& \int_{-\infty}^{\infty} \int_0^{\infty} \varphi(z,w) \left[ \int_{-\infty}^{\infty} \int_{-\infty}^{\infty} e^{-i(\xi^2-\eta^2)t+ix\xi-y|\eta|-i\xi z-|\eta| w} \,d\eta \, d\xi \right] \, dw \, dz \\
                    &=& 4 \int_0^{\infty} e^{-y\eta} \, \int_{-\infty}^{\infty} e^{ix\xi} \, \left[\int_{-\infty}^{\infty} \int_0^{\infty} e^{-i\xi z-\eta w} e^{-i(\xi^2-\eta^2)t} \varphi(z,w) \, dw \, dz \right] \, d\xi \, d\eta \\
                    &=& C \int_0^{\infty} e^{-y\eta} \, \F_{\xi}^{-1} \left[\int_{-\infty}^{\infty} \int_0^{\infty} e^{-i\xi z-\eta w} e^{-i(\xi^2-\eta^2)t} \varphi(z,w) \, dw \, dz \right] (x) \, d\eta \\
                    &=& C \int_0^{\infty} e^{-y\eta} \, \F_{\xi}^{-1} \left[\F_z \left[\int_0^{\infty} e^{-w\eta} e^{-i(\xi^2-\eta^2)t} \varphi(z,w) \, dw \right] (\xi) \right] (x) \, d\eta\, , \\
\end{eqnarray*}
we have
\begin{align}
\left\|\widetilde{\mathcal{K}}(t)\varphi \right\|& _{L^2_{xy}(\R \times \R^+)}^2  = \int_{-\infty}^{\infty} \left\|\widetilde{\mathcal{K}}(t)\varphi \right\|_{L^2_y (\R^+)}^2 \, dx \nonumber \\
                    &= \int_{-\infty}^{\infty} \left\|\int_0^{\infty} e^{-y\eta} \, \F_{\xi}^{-1} \left[ e^{-i(\xi^2-\eta^2)t} \int_0^{\infty} e^{-w\eta} \F_z \left[\varphi(z,w) \right] (\xi) \, dw \right] (x) \, d\eta \, \right\|^2_{L^2_y (\R^+)} \, dx \nonumber \\
                    &\le \int_{-\infty}^{\infty} \left\|\F_{\xi}^{-1} \left[ e^{-i(\xi^2-\eta^2)t} \int_0^{\infty} e^{-w\eta} \F_z \left[\varphi(z,w) \right] (\xi) \, dw \right] (x) \, \right\|^2_{L^2_{\eta} (\R^+)} \, dx  \nonumber \\
                    &= \int_0^{\infty} \left\|\F_{\xi}^{-1} \left[ e^{-i(\xi^2-\eta^2)t} \int_0^{\infty} e^{-w\eta} \F_z \left[\varphi(z,w) \right] (\xi) \, dw \right] (\cdot) \, \right\|^2_{L^2_x (\R)} \, d\eta \nonumber \\
                    &= \int_0^{\infty} \left\|\int_0^{\infty} e^{-w\eta} \F_z \left[\varphi(z,w) \right](\cdot) \, dw \, \right\|^2_{L^2_{\xi} (\R)} \, d\eta \nonumber \\
                    &= \int_{-\infty}^{\infty} \left\|\int_0^{\infty} e^{-w\eta} \F_z \left[\varphi(z,w) \right](\xi) \, dw \, \right\|^2_{L^2_{\eta} (\R^+)} \, d\xi \nonumber \\
                    &\le \left\|\F_z \left[\varphi(z,w) \right] \right\|^2_{L^2_{\xi w} (\R \times \R^+)} =\|\varphi\|^2_{L^2_{z,w} (\R \times \R^+)}\, .\label{3.23}
\end{align}
By (\ref{3.22}), (\ref{3.23}) and interpolation, we obtain
                \begin{equation}
                    \left\|\widetilde{\mathcal{K}}(t)\varphi \right\|_{L^r_{xy}(\R \times \R^+)} \lesssim t^{-2\left(\frac{1}{2}-\frac{1}{r}\right)} \|\varphi\|_{L^{r'}(\R \times \R^+)}\, . \label{3.24}
                \end{equation}
Then, for $\phi \in C_c \left([0,T]; \, L_{xy}^{r'}(\R \times \R^+)\right)$, a classical density argument leads to
\begin{equation}
\left\|\widetilde{\mathcal{K}}(t)\phi \right\|_{L^r_{xy}(\R \times \R^+)} \lesssim t^{-2\left(\frac{1}{2}-\frac{1}{r}\right)} \|\phi(t)\|_{L^{r'}(\R \times \R^+)}\label{3.25}
\end{equation}
if $t> 0 $ is fixed. Thus,
                \begin{align*}
                    \left\|\int_0^t \widetilde{\mathcal{K}}(t-s)\phi(s) \, ds \, \right\|_{L^r_{xy}} &\le C \int_0^T |t-s|^{-2 \left(\frac{1}{2}-\frac{1}{r}\right)} \|\phi(s)\|_{L^{r'}_{xy}} \, ds
                    = C \int_0^T |t-s|^{\frac{-2}{q}} \|\phi(s)\|_{L^{r'}_{xy}} \, ds\, .
                \end{align*}
By Riesz potential inequality,
                \begin{eqnarray*}
                    \left\|\int_0^{(\cdot)} \widetilde{\mathcal{K}}(\cdot-s)\phi(s) \, ds \, \right\|_{L^q_t \left([0,T]; \, L^r_{xy}\right)} &\le& \left\|\int_0^T \widetilde{\mathcal{K}}(\cdot-s)\phi(s) \, d s \, \right\|_{L^q_t \left([0,T]; \, L^r_{xy} (\R \times \R^+)\right)} \nonumber \\
                    &\lesssim& \|\phi\|_{L^{q'}_t \left([0,T]; \, L^{r'}_{xy} (\R \times \R^+)\right)}
                \end{eqnarray*}
or
                \begin{equation} 
                    \left\|\int_0^T \int_{-\infty}^{\infty} \int_0^{\infty} \phi(z,w,s) \widetilde{K}_{t-s} (x,y,z,w) \, dz \, dw \, ds \, \right\|_{L^q_t \left([0,T]; \, L^r_{xy}\right)} \lesssim \|\phi\|_{L^{q'}_t \left([0,T]; \, L^{r'}_{xy}\right)}\, .\label{3.26}
                \end{equation}
To find the estimate of $\|\mathcal{K}_T\phi\|^2_{L^2}$, by (\ref{3.26}) we have
\begin{align}
                   \|\mathcal{K}_T& \phi\|_{L^2(\R \times \R^+)}^2 = \int_0^T \int_0^T \left(\int_{-\infty}^{\infty} \int_{-\infty}^{\infty} \int_0^{\infty} \int_0^{\infty} \overline{\phi(x,y,s)} \phi(z,w,\sigma) \widetilde{K}_{s-\sigma} (x,y,z,w) dw \, dy \, dz \, dx \right) \, ds \, d\sigma \nonumber \\
                    &= \int_0^T  \int_{-\infty}^{\infty} \int_0^{\infty} \overline{\phi(x,y,s)} \left[\int_0^T \left(\int_{-\infty}^{\infty} \int_0^{\infty} \phi(z,w,\sigma) \widetilde{K}_{s-\sigma} (x,y,z,w) \, dw \, dz \right) \, d\sigma \right] \, dy \, dx \, ds \nonumber\\
                    &\le C \|\phi\|_{L^{q'}([0,T]; L^{r'}(\R \times \R^+))} \cdot \left\|\int_0^T \int_{-\infty}^{\infty} \int_0^{\infty} \phi(z,w,\sigma) \widetilde{K}_{s-\sigma} (x,y,z,w) \, dz \, dw \, d\sigma \, \right\|_{L^q([0,T]; L^r(\R \times \R^+))} \nonumber\\
                    &\le C \|\phi\|_{L^{q'}([0,T]; L^{r'}(\R \times \R^+))} \cdot \|\phi\|_{L^{q'}([0,T]; L^{r'}(\R \times \R^+))} = C \|\phi\|_{L^{q'}([0,T]; L^{r'}(\R \times \R^+))}^2\, .\label{3.27}
                \end{align}
Finally, from (\ref{3.17}) and (\ref{3.15}), it is deduced that
                \begin{equation*}
                    \int_0^T \left<\mathcal{K}(t)\psi, \phi(t)\right>_{L^2_{xy} (\R \times \R^+)} \, dt \lesssim \|\psi\|_{L^2(\R \times \R^+)} \cdot \|\phi\|_{L^{q'}_t \left([0,T]; \, L^{r'}_{xy}(\R \times \R^+)\right)}
                \end{equation*}
or
                \begin{equation*}
                    \|\mathcal{K}(t)\psi\|_{L^q([0,T]; L^r(\R \times \R^+))} \le C \|\psi\|_{L^2(\R \times \R^+)}\, ,
                \end{equation*}
which yields
                \begin{equation*}
                    \|W_{b_2}[h]\|_{L^q_t\left(\R^+; \, L^r_{xy}(\R \times \R^+)\right)} = \|\mathcal{K}(t)\Psi_h\|_{L^q_t\left(\R^+; \, L^r_{xy}(\R \times \R^+)\right)} \le C \cdot \|\Psi_h\|_{L^2_{(\xi, \eta)}}\, .
                \end{equation*}
The proof of (\ref{3.12})  is then completed. Similar to (\ref{3.11.5}), we have
\begin{align*}
\|\Psi_h\|^2_{L^2} &= \|\widehat{\Psi_h}(\xi,\eta)\|^2_{L^2} \eqsim \int_{-\infty}^{\infty} \int_0^{\infty} \eta^2 \left|\widetilde{h}\big (\xi,-i(\xi^2-\eta^2)\big ) \right|^2 \, d\eta \, d\xi \\
&\eqsim \int_{-\infty}^{\infty} \int_0^{\infty} \eta^2 \left|\int_0^{\infty} e^{i(\xi^2-\eta^2)t} \int_{-\infty}^{\infty} e^{-i\xi x} h(x,t) \, dx \, dt\right|^2 d\eta \, d\xi \\
&= C \int_{-\infty}^{\infty} \int_0^{\infty} \sqrt{\beta} \left|\int_0^{\infty} \int_{-\infty}^{\infty} e^{-i\beta t} e^{i\xi^2t} e^{-i\xi x} h(x,t) \, dx \, dt \right|^2 d\beta \,d\xi\, ,
                \end{align*}
which implies
                \begin{equation}
                    \|W_{b_2}[h]\|_{L^q_t \left([0,T]; \, L^r_{xy}(\R \times \R^+)\right)} \lesssim \left\{\int_{-\infty}^{\infty} \int_0^{\infty} \sqrt{\beta} \left|\int_0^{\infty} \int_{-\infty}^{\infty} e^{-i\beta t} e^{i\xi^2t} e^{-i\xi x} h(x,t) \, dx \, dt \right|^2 d\beta \,d\xi \right\}^{\frac{1}{2}}\, .\label{3.28}
                \end{equation}

Next, we turn to study the case with $ s > 0$, i.e.,
                \begin{align}
                     \|W_{b_2}& [h]\|_{L^q_t \left([0,T]; \, W^{s,r}_{xy}(\R \times \R^+)\right)} \nonumber \\
                    &\lesssim \left\{\int_{-\infty}^{\infty} \int_{-\infty}^0 (1+|\beta|+\xi^2)^s \sqrt{|\beta|} \left|\int_0^{\infty} \int_{-\infty}^{\infty} e^{i\left(\beta+\xi^2\right)t} e^{-i\xi x} h(x,t) \, dx \, dt \right|^2 \, d\beta \, d\xi \right\}^{\frac{1}{2}}\, .\label{3.29}
                \end{align}
Write the multi-index $\alpha = (\alpha_1, \, \alpha_2)^T \in \Z^2$, i.e., $(\xi\eta)^{\alpha} = \xi^{\alpha_1} \eta^{\alpha_2}$, and let $\alpha_1$, $\alpha_2\ge 0$ with $|\alpha| = \alpha_1 + \alpha_2 = m$. Note that
\begin{eqnarray*}
                   D^{\alpha}_{xy} W_{b_2}[h] &\eqsim& \int_{-\infty}^{\infty} \int_0^{\infty} e^{-i(\xi^2-\eta^2)t+ix\xi-y\eta} \cdot [(i\xi)(-\eta)]^{\alpha} \cdot \eta \widetilde{h}\big (\xi, -i(\xi^2-\eta^2)\big ) \, d\eta \, d\xi \\
                   &=& \int_{-\infty}^{\infty} \int_0^{\infty} e^{-i(\xi^2-\eta^2)t+ix\xi-y\eta} \Psi_{h, \alpha} (\xi, \eta) \, d\eta \, d\xi\, ,
\end{eqnarray*}
where $\Psi_{h, \alpha} (\xi, \eta) = [(i\xi)(-\eta)]^{\alpha} \cdot \eta \widetilde{h}\left (\xi, -i(\xi^2-\eta^2)\right )$ for $\eta \ge 0$ and otherwise $\Psi_{h, \alpha}=0$.
Thus, $\left|(\xi\eta)^{\alpha}\right| \le C(m) \left ( 1 + |\xi|^2 + |\eta|^2 \right ) ^{m/2} $ and the inequality with $s=0$ imply that
\begin{align*}
                    \sum_{|\alpha|=m} & \left\|D^{\alpha}_{xy} W_{b_2}[h]\right\|^2_{L^q_t\left ([0,T]; \, L^r_{xy}(\R \times \R^+)\right )} \lesssim \sum_{|\alpha|=m} \|\Psi_{h, \alpha}\|^2_{L^2} \\
                   &\le C(m) \int_{-\infty}^{\infty} \int_0^{\infty} \left (1+ \beta+\xi^2\right )^m \sqrt{\beta} \left|\int_0^{\infty} e^{-i(\beta-\xi^2)t} \int_{-\infty}^{\infty} e^{-i\xi x} h(x,t) \, dx \, dt \right|^2 \,d\beta \, d\xi \\
                   &= C(m) \int_{-\infty}^{\infty} \int_{-\infty}^0 \left (1+ |\beta|+\xi^2\right )^m \sqrt{|\beta|} \left|\int_0^{\infty} e^{i(\beta+\xi^2)t} \int_{-\infty}^{\infty} e^{-i\xi x} h(x,t) \, dx \, dt \right|^2 \,d\beta \, d\xi\, .
\end{align*}
A simple interpolation argument leads to (\ref{3.29}).

To prove (\ref{3.9}),  let $\lambda = -\left(\beta+\xi^2\right)$, which gives $\beta = -\left(\lambda+\xi^2\right)$, then
\begin{align}
 \|W_b [h]& \|^2_{L^q_t\left ([0,T]; \, W^{s,r}_{xy}(\R \times \R^+)\right )} \nonumber \\
                &\le C \int_{-\infty}^{\infty} \int_{-\infty}^{\infty} \left (1+|\beta|+\xi^2\right )^s \sqrt{|\beta|} \left|\int_0^{\infty} e^{i(\beta+\xi^2)t} \int_{-\infty}^{\infty} e^{-i\xi x} h(x,t) \, dx \, dt \right|^2 \,d\beta \,d\xi \nonumber\\
                &= C \int_{-\infty}^{\infty} \int_{-\infty}^{\infty} \left(1+|\lambda+\xi^2|+\xi^2\right)^s \sqrt{|\lambda+\xi^2|} \left|\int_0^{\infty} e^{-i\lambda t} \int_{-\infty}^{\infty} e^{-i\xi x} h(x,t) \, dx \, dt\right|^2 \,d\lambda \,d\xi \nonumber\\
                &\leq C \int_{-\infty}^{\infty} \int_{-\infty}^{\infty} \left(1+|\lambda|+\xi^2\right)^s \sqrt{|\lambda+\xi^2|} \left|\int_0^{\infty} e^{-i\lambda t} \int_{-\infty}^{\infty} e^{-i\xi x} h(x,t) \, dx \, dt\right|^2 \,d\lambda \,d\xi \nonumber \\
                &\leq C \| h\|^2_{{\cal H}_s} \nonumber
            \end{align}
where $h$ can be first chosen as a smooth function of compactly supported with respect to $t$ in $[0,T]$ and then a density argument can be applied. Thus, we finish the proof of  (\ref{3.9}). This completes the proof of the proposition.
\end{proof}


%

Now, we derive the estimates for the solution $v(x,y,t) = W_{\R^2}(t)\phi (x,y)$ in (\ref{2.2}), where
        \begin{equation}
            W_{\R^2}(t) \phi(x,y) = \int_{-\infty}^{\infty} \int_{-\infty}^{\infty} e^{-i(\xi^2+\eta^2)t+i(\xi x + \eta y)} \widehat{\phi}(\xi,\eta) \, d\xi \, d\eta\, , \label{3.32}
        \end{equation}
        (see \emph{Section 2.2} of \cite{cazenave}).
The following estimates are needed later.
\begin{prop}\label{prop3.4} 
        Let $\frac{1}{q}+\frac{1}{r} = \frac{1}{2}$ where $r\in [2, \infty)$.
\begin{align}
& \left\|W_{\R^2}\phi\right\|_{L^q_t \left([0,T]; \, W^{s,r}_{xy}(\R^2)\right)} \lesssim \left\|\phi\right\|_{H^s_{xy}(\R^2)}\, ,\label{3.33}\\
& \left\|W_b \left[(W_{\R^2}\phi)\big |_{y=0}\right] \, \right\|_{L^q_t \left([0,T]; \, W^{s,r}_{xy}(\R^2)\right)} \lesssim \|\phi\|_{H^s(\R^2)}\, .\label{3.34}
\end{align}

\end{prop}
\begin{proof}
         (\ref{3.33})  is proved in \emph{Section 2.3} of \cite{cazenave} as the classic Strichartz estimate and
(\ref{3.34}) follows from Lemma \ref{lemm2.2} and Proposition \ref{prop3.2}.


\end{proof}

Finally, we derive the estimates for the nonhomogeneous terms in (\ref{2.3}). Let $f: \, \R^2 \times [0,T] \to \C$ be an extension of $g$, i.e. $f(x,y,t) = g(x,y,t)$ for $y\ge0$. We will show that the solution $ u = \Phi_f $ of  (\ref{2.3})
satisfies some estimates needed later, where
it is well-known that
        \begin{equation}
            \Phi_f(x,y,t) = i \left(\int_0^t W_{\R^2}(t-\tau)f(\tau) \, d\tau \right) (x,y) = i \int_0^t \int_{-\infty}^{\infty} \int_{-\infty}^{\infty} e^{-i(\xi^2+\eta^2)(t-\tau)+i(\xi x + \eta y)} \widehat{f^{xy}}(\xi,\eta,\tau) \, d\xi \, d\eta \, d\tau\, .\label{3.37}
        \end{equation}
\begin{prop}\label{prop3.5} 
Let $\frac{1}{q}+\frac{1}{r} = \frac{1}{2}$ and  $\frac{1}{\gamma}+\frac{1}{\rho} = \frac{1}{2}$ where $r$, $\rho\in [2, \infty)$. Then,
\begin{align}
& \left\|\Phi_f \, \right\|_{L^q_t \left([0,T]; \, W^{s,r}_{xy}(\R^2)\right)} \lesssim \left\|f \, \right\|_{L^{\gamma'}_t \left([0,T]; \, W^{s,\rho'}_{xy}(\R^2)\right)} \, ,\label{3.38}\\
& \left\|W_b \left[\Phi_f|_{y=0}\right] \, \right\|_{L^q_t \left([0,T]; \, W^{s,r}_{xy}(\R^2)\right)} \lesssim \|f\|_{L^1_t \left([0,T]; \, H^s_{xy}(\R^2)\right)}\, .\label{3.39}
\end{align}
\end{prop}
 \begin{proof}
The proof of (\ref{3.38}) can be found as the Strichartz estimate in \emph{Section 2.3} of \cite{cazenave}.
For (\ref{3.39}),  we first show that
\begin{align}
W_b \left[\Phi_f|_{y=0}\right] = i \left(\int_0^t W_b\left [ [W_{\R^2}(t)f(\tau)]\Big |_{y=0} \right ](t-\tau, x, y )   \, d\tau \right)\, .\label{3.40}
\end{align}
Here, we note that $f(x, y, \tau)$ is a function of $(x, y)$ with $\tau $ as a fixed parameter and
$W_{R^2} (t) f (\tau)  $ is a function of $( x, y, t )$ (again $\tau$ fixed), which implies that
 $ W_b [ W_{R^2} (t) f (\tau) |_{y=0} ] (t, x, y) $ is a function of $( x, y, t)$ ($\tau$ as parameter).
Let $v $ be the right side of (\ref{3.40}). Then, $v$ satisfies
\begin{align*}
 & iv_t +  v_{xx} + v_{yy} =  - W_b\left [ W_{\R^2}(t )f(\tau )\Big |_{y=0}\right ] \Big | _{ \tau = t} ( 0 , x, y) + i  \int_0^t \bigg ( i \left ( W_b\left [ \left [W_{\R^2}(t)f(\tau) \right ]\Big |_{y=0}\right ] (t-\tau , x, y) \right )_t   +\\
&\qquad
\left ( W_b\left [ \left [W_{\R^2}(t)f(\tau)\right ]\Big |_{y=0}\right ] (t-\tau , x, y)\right )_{xx}
 + \left ( W_b\left [ \left [W_{\R^2}(t)f(\tau)\right ]\Big |_{y=0} \right ] (t-\tau , x, y) \right )_{yy}  \bigg )
 \, d\tau \\
  & = - W_b\left [ W_{\R^2}(t )f(\tau )\Big |_{y=0}\right ] \Big | _{ \tau = t} ( 0 , x, y) = 0 \, ,\\
 & v|_{t =0 } = 0 \, ,\quad v |_{y = 0 } = i \int_0^t \left [ W_{\R^2}(t-\tau)f(\tau)\right ]\Big |_{y=0}   \, d\tau\, .
\end{align*}
Thus, by the uniqueness of solution for linear homogeneous Schr\"odinger equation with zero
initial condition and nonhomogeneous boundary condition,
$$
v = W_b \left [ i \int_0^t \left [ W_{\R^2}(t-\tau)f(\tau)\right ] \, d \tau \Big |_{y=0} \right ]=W_b \left[\Phi_f|_{y=0}\right]\, .
$$
For $ s < 1/2$, by the Minkowski inequality and Proposition 3.5, we have
\begin{align*}
 \left\|W_b \left[\Phi_f|_{y=0}\right] \, \right\|&_{L^q_t\left([0,T]; \, W_{xy}^{s,r}(\R \times \R^+)\right)}\\
 & =   \left\| \int_0^T  \zeta_{[0, t]} (\tau) W_b\left [ [W_{\R^2}(t)f(\tau)]\Big |_{y=0} \right ](t-\tau, x, y ) \, d\tau
\, \right\|_{L^q_t\left([0,T]; \, W_{xy}^{s,r}(\R \times \R^+)\right)} \\
&\leq \int_0^{T} \left\| \zeta_{[\tau, \infty)} (t)W_b\left [ [W_{\R^2}(t)f(\tau)]\Big |_{y=0} \right ](t-\tau, x, y )
\, \right\|_{L^q_t\left([0,T]; \, W_{xy}^{s,r}(\R \times \R^+)\right)} \, d\tau \\
&=  \int_0^{T} \left\|  W_b\left [ [W_{\R^2}(t)f(\tau)]\Big |_{y=0} \right ](t, x, y )
\, \right\|_{L^q_t\left([0,T-\tau ]; \, W_{xy}^{s,r}(\R \times \R^+)\right)} \, d\tau \\
& \leq C\int_0^{T} \left\| f(\tau, x, y )
\, \right\|_{H_{xy}^{s}(\R^2)} \, d\tau =  \|f\|_{L^1_t \left([0,T]; \, H^s_{xy}(\R^2)\right)}\, ,
 \end{align*}
where $\zeta $ is a characteristic function, which gives \eqref{3.39}. For $ s > 1/2$ and s an integer, by the form of $W_b [h]$ defined in Proposition \ref{prop3.1},
we can take the derivatives inside $W_b$, use the form of $\Phi_f ( x, y, t) $ in \eqref{3.37} to move the derivatives to $f$, and then apply \eqref{3.39} with $s=0$. A classical  interpolation theorem yields  \eqref{3.39} for non-integer $s > 0$.
Thus, the proof of the estimates with  the nonhomogeneous term is completed.
    \end{proof}

\section{Local Well-posedness}

In the following, we will show by {\em contraction mapping principle}  that the integral equation
    \begin{align}
        u(x,y,t) =& W_b \left[h - W_{\R^2}(\cdot) \phi \big |_{y=0} - i \left(\int_0^{\cdot} W_{\R^2}(\cdot-\tau) f(\tau) \, d\tau \right)\Big |_{y=0} \right](x,y,t) \nonumber \\
        \quad & + W_{\R^2}(t) \phi (x,y) + \, i \left(\int_0^t W_{\R^2}(t-\tau) f(\tau) \, d\tau \right) (x,y) = \mathcal{A}[u](x,y,t)\label{4.0}
    \end{align}
has a solution in some function spaces with $H^s (\R \times \R^+ )$, i.e., $ \mathcal{A}[u](x,y,t)$ has a fixed point, where $f(u) = \lambda |u|^{p-2} u \  \text{for } p \ge 3 \text{ and } (x,y,t) \in \R \times \R^+ \times (0,T)$.

Define some function spaces as follows.
Let $(q,r)$ be an admissible pair and
    \begin{equation*}
        \X_T^s := C_t \left([0,T]; \, H^s_{xy}(\R \times \R^+)\right) \bigcap L^q_t \left([0,T]; \, W^{s,r}_{xy}(\R \times \R^+)\right)
    \end{equation*}
with $ \|u\|_{\X_T^s} = \sup_{t\in[0,T]} \|u(t)\|_{H^s_{xy}(\R \times \R^+)} + \|u\|_{L^q_t\left([0,T]; \, W^{s,r}_{xy}(\R \times \R^+)\right)}$, as well as
    \begin{equation*}
        \mathcal{Z}_T^s := C_t \left([0,T]; \, H^s_{xy}(\R \times \R^+)\right)
    \end{equation*}
with the regular $\sup$-norm in the time variable, i.e., $\|u\|_{\mathcal{Z}_T^s} = \sup_{t\in[0,T]} \|u(t)\|_{H^s_{xy}(\R \times \R^+)}$.

For some $M>0$, define closed balls in $\X_T^s$ and $\mathcal{Z}_T^s$ as
    \begin{equation*}
        B_M^{\X^s} := \left\{u: \, \|u\|_{\X_T^s} \le M\right\} \quad\text{and } \quad B_M^{\mathcal{Z}^s} := \left\{u: \, \|u\|_{\mathcal{Z}_T^s} \le M\right\}\, .
    \end{equation*}
First, we introduce the following Lemmas.
\begin{lemma}\label{lemm4.1} 
Let $0<s<1$ and $\alpha$ be a nonnegative multi-index with $|\alpha|=s$. If $u: \R^2 \to \C$ and $f\in C^1(\C)$, then
        $$
            \|D^{\alpha} f(u)\|_{L^{r_2}} \lesssim \|f'(u)\|_{L^{r_1}} \|D^{\alpha} u\|_{L^r}
       $$
for $\frac{1}{r_2} = \frac{1}{r_1} + \frac{1}{r}$ with $1<r$, $r_1$, $r_2<\infty$.
In particular, if $|f'(u)|$ is uniformly bounded, then
       $$
            \|D^s f(u)\|_{L^r} \lesssim \|f'(u)\|_{L^{\infty}} \|D^s u\|_{L^r}\, .
       $$
\end{lemma}
\noindent This lemma is the chain rule for fractional derivatives (see Proposition 3.1 in \cite{Weinstein_KdV_DP_SA}).
\begin{lemma}\label{lemm4.2} 
Let $0<s<1$, $u$, $v: \R^2 \to \C$. Then
$$
            \|D^{\alpha} (uv)\|_{L^{r_2}} \lesssim \|v\|_{L^{r_1}} \|D^{\alpha} u\|_{L^r} + \|u\|_{L^{\tilde{r}_1}} \|D^{\alpha} v\|_{L^{\tilde{r}}}
$$
for $\frac{1}{r_2} = \frac{1}{r_1} + \frac{1}{r} = \frac{1}{\tilde{r}_1} + \frac{1}{\tilde{r}}$ with $1<r$, $r_1$, $\tilde{r}$, $\tilde{r}_1$, $r_2<\infty$.
\end{lemma}
\noindent This lemma is the product rule for fractional derivatives (see Proposition 3.3 in \cite{Weinstein_KdV_DP_SA}).
\smallskip

Lemmas \ref{lemm4.1} and \ref{lemm4.2} are valid for functions in $\R^2$. For functions in $\R\times \R^+$, an extension operator
$E$ such that $Eu \in W^{s, p } (\R^2 ) $ and $Eu | _{\R \times \R^+ } = u $ for any $ u \in W^{s, p } (\R\times \R^+ )$ satisfying
$$
\| Eu \|_{ W^{s, p } (\R^2 )} \leq C \| u \|_{ W^{s, p } (\R\times \R^+ )}\,
$$
can be used (see Chapter 5, \cite{Adams_Sobo}). Therefore, Lemmas \ref{lemm4.1} and \ref{lemm4.2} can be used for functions in $\R\times \R^+ $ except that the derivatives on the right hand sides of inequalities are replaced by $W^{|\alpha|, p }$-norms of $u$ and $v$.

\begin{thm}\label{theo4.3} 
        Let $\mu >0 $ be a constant so that $\displaystyle \|\varphi\|_{H^s(\R \times \R^+)} +
        \|h\|_{{\cal H}^s (0,T)} \le \mu$
        and $\varphi, h$ satisfy some compatibility conditions.
        \begin{itemize}
        \item[(a)]
            For $0 \le s <1$ and $ 3 \leq p < (4-2s)/(1-s) $, there exists a $T>0$ such that for a given admissible pair $(q,r)$, there is a unique solution $u \in \X^s_T$ of (\ref{4.0}). In particular, we have
            \begin{align}
                & \left\|\mathcal{A}[u]-\mathcal{A}[v]\right\|_{\X^s_T} \le (1/2)  \, \|u-v\|_{\X^s_T} \, , \label{4.1} \\
                & \left\|\mathcal{A}[u]\right\|_{\X^s_T} \le M \label{4.2}
            \end{align}
           for any $u$ and $v \in B_M^{\X^s}$. Moreover, if $ p = (4-2s)/(1-s)$, then $\mu$ must be small.
        \item[(b)]
            For $s=1$, there exists a $ T>0$ such that for any admissible pair $(q,r)$
            with $ r > 2$ the integral equation (\ref{4.0}) has a unique solution $u \in \X^1_T$. In particular, we have
            \begin{align}
                & \left\|\mathcal{A}[u]-\mathcal{A}[v]\right\|_{\X^1_T} \le (1/2) \, \|u-v\|_{\X^1_T} \, , \label{4.3} \\
                & \left\|\mathcal{A}[u]\right\|_{\X^1_T} \le M \label{4.4}
            \end{align}
for any $u$ and $v \in B_M^{\X^1}$.
        \item[(c)]
            For $s>1$, we assume that $p\ge s+1$ if $s\in\Z$ or $p\ge[s]+2$ if $s\notin\Z$ only when $p$ is not an even integer. There exists a $ T>0$ such that the integral equation (\ref{4.0}) has   a unique solution $u \in \mathcal{Z}^s_T = C_t \left([0,T]; \, H^s_{xy}(\R \times \R^+)\right)$. Also, for $u$ and $v \in B_M^{\mathcal{Z}^s}$
            \begin{align}
                & \left\|\mathcal{A}[u]-\mathcal{A}[v]\right\|_{\mathcal{Z}^0_T} \le (1/2) \, \|u-v\|_{\mathcal{Z}^0_T} \, , \label{4.5} \\
                &  \left\|\mathcal{A}[u]\right\|_{\mathcal{Z}^s_T} \le M \, .\label{4.6}
            \end{align}
        \end{itemize}
\end{thm}
\begin{proof}
For $ 0\leq s < 1$, let
\begin{equation}
                r = \frac{2(p-1)}{1+s(p-2)} \, \text{, } \qquad q = \frac{2(p-1)}{(1-s)(p-2)} \, , \label{4.8}
\end{equation}
and assume that
\begin{equation} 
                3 \le p \le \frac{4-2s}{1-s}\, .\label{4.9}
\end{equation}
It is easy to verify that $r>2$ and $(q,r)$ is an admissible pair.
We will use the contraction mapping theorem.
First, let $3 \le p < \frac{4-2s}{1-s}$, the subcritical case.
In \cite{cazenave} (the proof of \emph{Theorem 4.6.1}), we find the estimate
                \begin{equation}
                    \left\|f(u)\right\|_{L^{q'}_t\left([0,T]; L^{r'}_{xy}(\R \times \R^+)\right)} \lesssim T^{\frac{q-p}{q}} \|u\|_{L^q_t\left([0,T]; L^r_{xy}(\R \times \R^+)\right)}^{p-1}\, . \label{4.10}
                \end{equation}
Let $\alpha \in \R^+ \times \R^+ $ be a multi-index such that $|\alpha|=s$. We know that $u: \R^2 \to \R^2$ and $f\in C^1(\C)$. For $0\le s < 1$, the chain rule for the fractional derivative here gives
                \begin{equation*}
                    \|D^{\alpha} f(u)\|_{L^{r_2}} \lesssim \|f'(u)\|_{L^{r_1}} \| u\|_{W^{s,r}}
                \end{equation*}
for $\frac{1}{r_2} = \frac{1}{r_1} + \frac{1}{r}$ with $1<r$, $r_1$, $r_2<\infty$.
Let $u_R$ and $u_I$ be the real and imaginary parts of $u$, respectively, with which $f(u) = \lambda |u|^{p-2} u = \lambda u_R \left(u_R^2 + u_I^2\right)^{\frac{p}{2}-1} + i \lambda u_I \left(u_R^2 + u_I^2\right)^{\frac{p}{2}-1}$. Therefore
\begin{equation} 
                    f'(u) = \lambda
                    \left(
                    \begin{array}{ll}
                        (p-2)|u|^{p-4} u_R^2 + |u|^{p-2} & (p-2)|u|^{p-4} u_R u_I \\
                        (p-2)|u|^{p-4} u_R u_I & (p-2)|u|^{p-4} u_I^2 + |u|^{p-2}
                    \end{array}
                    \right)\, ,\label{4.11}
\end{equation}
which yields that $|f'(u)| \lesssim |u|^{p-2} < \infty$ for $p\ge2$.  In fact, for $k\in\Z^+$, we continue differentiating to have higher order of derivatives of $f$. If $p\ge k+1$, it is observed that $(k-1)$-th derivative $f^{(k-1)}$ can be composed of terms in the form $|u|^{a_1}u_R^{a_2}u_I^{a_3}$ where $a_1+a_2+a_3=p-k$. Through a simple calculation, $$\partial_{u_R} \left(|u|^{a_1}u_R^{a_2}u_I^{a_3}\right) = a_1|u|^{a_1-2}u_R^{a_2+1}u_I^{a_3} + a_2|u|^{a_1}u_R^{a_2-1}u_I^{a_3}.$$ Thus $f^{(k)}$ is term of $|u|^{b_1}u_R^{b_2}u_I^{b_3}$ where $b_1+b_2+b_3=p-k-1$. Moreover, for $p\ge k+1$
        \begin{equation*}\label{NLSE_IB_2d_nonlin_Derivative_kth_est} 
            |f^{(k)}(u)| \lesssim |u|^{p-k-1} \, .
        \end{equation*}

For $u \in B_M^{\X^s}$,
                \begin{align*}
                     \left\|D^{\alpha} f(u)(t)\right\|_{L^2} &\lesssim \left\|f'(u)(t)\right\|_{L^{\frac{2r}{r-2}}} \| u(t)\|_{W^{s,r}} \\
                    &\lesssim \left\||u(t)|^{p-2}\right\|_{L^{\frac{2r}{r-2}}} \| u\|_{W^{s,r}} = \|u(t)\|_{L^{\frac{2r(p-2)}{r-2}}}^{p-2} \| u(t)\|_{W^{s,r}}\, .
                \end{align*}
Since $r = \frac{2(p-1)}{1+s(p-2)}$ by (\ref{4.8}), then $\frac{r-2}{2r(p-2)} = \frac{1}{r} - \frac{s}{2}$. Thus, according to Gagliardo-Nirenberg inequality for fractional derivatives \cite{HMOW2011}, one can obtain that
                \begin{equation*}
                    \left\|D^{\alpha} f(u)(t)\right\|_{L^2} \lesssim \| u(t)\|_{W^{s,r}}^{p-2} \cdot \| u(t)\|_{W^{s,r}}\, .
                \end{equation*}
Because (\ref{4.9}) exactly leads to $q \ge \frac{q(p-2)}{q-1}$, it follows that by taking norm with respect to $t$,
\begin{align}
                     \left\|D^{\alpha} f(u)\right\|& _{L^1_t\left([0,T]; L^2_{xy}(\R \times \R^+)\right)} \lesssim \| u\|_{L^{\frac{q(p-2)}{q-1}}_t\left([0,T]; W^{s,r}_{xy}(\R \times \R^+)\right)}^{p-2} \cdot \| u\|_{L^q_t\left([0,T]; W^{s,r}_{xy}(\R \times \R^+)\right)} \nonumber \\
                    &\lesssim T^{\frac{q-1}{q}-\frac{p-2}{q}} \| u\|_{L^q_t\left([0,T]; W^{s,r}_{xy}(\R \times \R^+)\right)}^{p-1} = T^{\frac{q-p+1}{q}} \| u\|_{L^q_t\left([0,T]; W^{s,r}_{xy}(\R \times \R^+)\right)}^{p-1}\, ,\label{4.14}
\end{align}
Also it is straightforward to see that
                \begin{equation}
                    \left\|f(u)\right\|_{L^1_t\left([0,T]; L^2_{xy}(\R \times \R^+)\right)} \lesssim T^{\frac{q-p+1}{q}} \|u\|_{L^q_t\left([0,T]; L^r_{xy}(\R \times \R^+)\right)}^{p-1}\, .\label{4.15}
                \end{equation}
Hence,  (\ref{4.14}) and (\ref{4.15}) imply
\begin{align}
                    \left\|f(u)\right\|_{L^1_t\left([0,T]; H^s_{xy}(\R \times \R^+)\right)} \lesssim& T^{\frac{q-p+1}{q}} \|u\|_{L^q_t\left([0,T]; W^{s,r}_{xy}(\R \times \R^+)\right)}^{p-1} \nonumber \\
                    =& T^{1 - \frac{(1-s)(p-2)}{2}} \|u\|_{L^q_t\left([0,T]; W^{s,r}_{xy}(\R \times \R^+)\right)}^{p-1}\, .\label{4.16}
\end{align}

Let $u, v \in \X^s_T$ and
$w = u \cdot \theta + v \cdot (1-\theta)$ for $\theta\in[0,1]$. Then
    \begin{equation*}
        |u|^{p-2}u - |v|^{p-2}v = \int_0^1 \left[\dfrac{p}{2} |w|^{p-2} + \left(\dfrac{p}{2}-1\right) |w|^{p-4} w^2\right] \cdot (u-v) \, d\theta \, .
    \end{equation*}
By Lemma \ref{lemm4.2},
    \begin{align*}
        \Big \|D^{\alpha} \big [|u|^{p-2}u(t) - & |v|^{p-2}v(t)\big ] \Big \|_{L^{r'}} \\
        =& \left\|\int_0^1 D^{\alpha}\left[\frac{p}{2} |w|^{p-2} + \left(\frac{p}{2}-1\right) |w|^{p-4} w^2\right] \cdot (u-v) \, d\theta \right\|_{L^{r'}} \\
        \le& \sup_{\theta\in[0,1]} \left\|D^{\alpha}\left[\left(\frac{p}{2} |w|^{p-2} + \left(\frac{p}{2}-1\right) |w|^{p-4} w^2\right) \cdot (u-v)\right] \right\|_{L^{r'}} \\
        \le& \sup_{\theta\in[0,1]} \left\| \left[\frac{p}{2} |w|^{p-2} + \left(\frac{p}{2}-1\right) |w|^{p-4} w^2\right] \right\|_{W^{s, r_1}} \cdot \|u-v\|_{L^{r_2}} \\
        & + \sup_{\theta\in[0,1]} \left\|\frac{p}{2} |w|^{p-2} + \left(\frac{p}{2}-1\right) |w|^{p-4} w^2 \right\|_{L^{\frac{r}{r-2}}} \cdot \| u-v\|_{W^{s,r}} \, .
    \end{align*}
For each $t\in[0,T]$ and the admissible pair $(q,r)$ in (\ref{4.8}), we notice that $\frac{1}{r} - \frac{s}{2} = \frac{1-s}{2(p-1)}$. Therefore $W^{s,r} \hookrightarrow L^{\frac{2(p-1)}{1-s}}$. Let $r_2 = \frac{2(p-1)}{1-s}$ so that $$\dfrac{1}{r_1} = \frac{1}{2} - \frac{1}{r_2} = \frac{p+s-2}{2(p-1)} = \frac{1}{r} + \frac{(1-s)(p-3)}{2(p-1)} \, .$$ Moreover, it is clear that $\frac{2\rho(p-2)}{\rho-2} = \frac{2(p-1)}{1-s}$. Thus for $p>3$, we note that the derivative of $\frac{p}{2} |w|^{p-2} + \left(\frac{p}{2}-1\right) |w|^{p-4} w^2$ is bounded by $|w|^{p-3}$. By Lemma \ref{lemm4.1},
    \begin{align*}
         \Big \|D^{\alpha} &\big [|u|^{p-2}u(t) -  |v|^{p-2}v(t)\big ] \Big \|_{L^2} \\
        \lesssim& \sup_{\theta\in[0,1]} \left\||w(t)|^{p-3} \right\|_{L^{\frac{2(p-1)}{(1-s)(p-3)}}} \cdot \| w(t)\|_{W^{s,r}} \cdot \|u(t)-v(t)\|_{L^{\frac{2(p-1)}{1-s}}} \\
        & + \sup_{\theta\in[0,1]} \|w(t)\|_{L^{\frac{r(p-2)}{r-2}}}^{p-2} \cdot \| u(t)-v(t)\|_{W^{s,r}} \\
        =& \sup_{\theta\in[0,1]} \|w(t)\|_{L^{\frac{2(p-1)}{1-s}}}^{p-3} \cdot \| w(t)\|_{W^{s,r}} \cdot \|u(t)-v(t)\|_{L^{\frac{2(p-1)}{1-s}}} + \sup_{\theta\in[0,1]} \|w(t)\|_{L^{\frac{2(p-1)}{1-s}}}^{p-2} \cdot \| u(t)-v(t) \|_{W^{s,r}} \\
        \lesssim& \left(\| u(t)\|_{W^{s,r}}^{p-2} + \| v(t)\|_{W^{s,r}}^{p-2}\right) \cdot \|u(t)-v(t)\|_{W^{s,r}}\, .
    \end{align*}
For $p=3$, we see that $\frac{1}{2} = \frac{1+s}{4} + \frac{1-s}{4} = \frac{1}{r} + \frac{1-s}{4}$ and $W^{s,r} \hookrightarrow L^{\frac{4}{1-s}}$. By Lemma (\ref{lemm4.1}), it is shown that
    \begin{align*}
         & \left\|D^{\alpha} \left[|u|^{p-2}u(t) - |v|^{p-2}v(t)\right] \right\|_{L^2} \\
        &\lesssim \sup_{\theta\in[0,1]} \left\| w(t) \right\|_{W^{s,r}} \cdot \|u(t)-v(t)\|_{L^{\frac{4}{1-s}}} + \sup_{\theta\in[0,1]} \|w(t)\|_{L^{\frac{4}{1-s}}} \cdot \| u(t)-v(t)\|_{W^{s,r}} \\
        &\lesssim \left(\| u(t)\|_{W^{s,r}}^{p-2} + \| v(t)\|_{W^{s,r}}^{p-2}\right) \cdot \|(u(t)-v(t))\|_{W^{s,r}} \, .
    \end{align*}
Because of (\ref{4.8}) and (\ref{4.9}), we have $q \ge \frac{q(p-2)}{q-1}$ and
\begin{align*}
                     & \left\| D^\alpha (f(u)-f(v)) \right\|_{L^1_t\left([0,T]; L^2_{xy}(\R \times \R^+) \right)} \\
                    &\quad \lesssim \left(\|u\|_{L^{\frac{q(p-2)}{q-1}}_t\left([0,T]; W^{s,r}_{xy}(\R \times \R^+) \right)}^{p-2} + \|v\|_{L^{\frac{q(p-2)}{q-1}}_t\left([0,T]; W^{s,r}_{xy}(\R \times \R^+) \right)}^{p-2} \right) \cdot \|u-v\|_{L^q_t\left([0,T]; W^{s,r}_{xy}(\R \times \R^+) \right)} \\
                    &\quad \le T^{\frac{q-p+1}{q}} \left(\|u\|_{L^q_t\left([0,T]; W^{s,r}_{xy}(\R \times \R^+) \right)}^{p-2} + \|v\|_{L^q_t\left([0,T]; W^{s,r}_{xy}(\R \times \R^+) \right)}^{p-2} \right) \cdot \|u-v\|_{L^q_t\left([0,T]; W^{s,r}_{xy}(\R \times \R^+) \right)}\\
                    &\quad  \lesssim T^{1 - \frac{(1-s)(p-2)}{2}} \left(\|u\|_{L^q_t\left([0,T]; W^{s,r}_{xy}(\R \times \R^+) \right)}^{p-2} + \|v\|_{L^q_t\left([0,T]; W^{s,r}_{xy}(\R \times \R^+) \right)}^{p-2} \right) \|u-v\|_{L^q_t\left([0,T]; W^{s,r}_{xy}(\R \times \R^+)\right)}\, .
\end{align*}

Now, we choose $u$ in $B_M^{\X^s}$, that is, $\|u\|_{\X^s_T} \le M$. We study the operator ${\mathcal{A}}(u) $ in (\ref{4.0}) with the estimate
                \begin{equation*}
                    \left\|\mathcal{A}[u]\right\|_{\X^s_T} = \left\|\mathcal{A}[u]\right\|_{L^{\infty}_t\left([0,T]; H^s_{xy}(\R \times \R^+)\right)} + \left\|\mathcal{A}[u]\right\|_{L^{q}_t\left([0,T]; W^{s,r}_{xy}(\R \times \R^+)\right)} \, .
                \end{equation*}
By applying \emph{Proposition~\ref{prop3.2}}, \emph{\ref{prop3.4}}, and \emph{\ref{prop3.5}}, we obtain
\begin{align*}
                    &  \left\|\mathcal{A}[u]\right\|_{\X^s_T} \\
                    &\quad \le \|W_b [h]\|_{\X^s_T} + \left\|W_{\R^2}\phi \right\|_{\X^s_T} + \left\|\Phi_f \right\|_{\X^s_T} + \left\|W_b \left[(W_{\R^2}\phi)\big | _{y=0}\right] \right\|_{\X^s_T} + \left\|W_b \left[(\Phi_f)\big |_{y=0}\right] \right\|_{\X^s_T} \\
                    &\quad \lesssim \|h\|_{{\cal H}^s (0,T)} + \left\|\phi \right\|_{H^s_{xy}} + \left\|f \right\|_{L^1_t \left([0,T]; H^s_{xy} \right)} \\
                    &\quad \le \mu + \left\|f \right\|_{L^1_t \left([0,T]; H^s_{xy} \right)} \\
                    &\quad {\lesssim} \mu + T^{1 - \frac{(1-s)(p-2)}{2}} \|u\|_{L^{q}_t\left([0,T]; W^{s,r}_{xy}(\R \times \R^+)\right)}^{p-1} \\
                    &\quad \le \mu + T^{1 - \frac{(1-s)(p-2)}{2}} \|u\|_{\X^s_T}^{p-1}\, ,
\end{align*}
where ${(\ref{4.16})}$ has been used.
Since $u \in B_M^{\X^s}$, for some $C_0>0$,
                \begin{equation}
                    \left\|\mathcal{A}[u]\right\|_{\X^s_T} \le C_0 \left(\mu + T^{1 - \frac{(1-s)(p-2)}{2}} M^{p-1} \right)\, .\label{4.17}
                \end{equation}
Then, select $M$ sufficiently large such that $M>C_0\mu$ and let $T$ be small enough with
                \begin{equation}
                    0 < T \le \left(\frac{M-C_0\mu}{C_0 M^{p-1}}\right)^{\frac{2}{2-(1-s)(p-2)}}\, .\label{4.18}
                \end{equation}
This implies that for $\|u\|_{\X^s_T} \le M$, $\left\|\mathcal{A}[u]\right\|_{\X^s_T} \le M$, which gives (\ref{4.2}).

Moreover, for $\left\|\mathcal{A}[u]-\mathcal{A}[v]\right\|_{\X^s_T}$  and the estimate (\ref{4.1}), let $\Phi_f(u)$ denote the operator on the nonlinearity $f$ given by $u$, where $f \in L^1_t \left([0,T]; H^s_{xy}\right)$.
If we choose $u$ and $v$ in $B_M^{\X^s}$, then
                \begin{align*}
                    & \left\|\mathcal{A}[u]-\mathcal{A}[v]\right\|_{\X^s_T} \\
                    &\quad \le \left\|\Phi_f(u)-\Phi_f(v) \right\|_{\X^s_T} + \left\|W_b \left[(\Phi_f(u)-\Phi_f(v))\big |_{y=0}\right] \right\|_{\X^s_T} \\
                    &\quad \lesssim \ \left\|f(u)-f(v) \right\|_{L^1_t \left([0,T]; H^s_{xy} \right)} \\
                    &\quad {\lesssim} \ + T^{1 - \frac{(1-s)(p-2)}{2}} \left(\|u(t)\|_{L^{q}_t\left([0,T]; W^{s,r}_{xy} \right)}^{p-2} + \|v(t)\|_{L^{q}_t\left([0,T]; W^{s,r}_{xy} \right)}^{p-2} \right) \|u-v\|_{L^{q}_t\left([0,T]; W^{s,r}_{xy} \right)} \\
                    &\quad \lesssim \ T^{1 - \frac{(1-s)(p-2)}{2}} M^{p-2} \|u-v\|_{L^{q}_t\left([0,T]; W^{s,r}_{xy} \right)} \\
                    &\quad \le \ T^{1 - \frac{(1-s)(p-2)}{2}} M^{p-2} \|u-v\|_{\X^s_T}\, ,
                \end{align*}
i.e., for some constant $C_1$,
                \begin{equation}
                    \left\|\mathcal{A}[u]-\mathcal{A}[v]\right\|_{\X^s_T} \le C_1 T^{1 - \frac{(1-s)(p-2)}{2}} M^{p-2} \|u-v\|_{\X^s_T}\, .\label{4.19}
                \end{equation}
Choose $T$ and $M$ satisfying (\ref{4.18}) and
$C_1 T^{1 - \frac{(1-s)(p-2)}{2}} M^{p-2} \le (1/2) ,$
which yields (\ref{4.1}), i.e. $$\left\|\mathcal{A}[u]-\mathcal{A}[v]\right\|_{\X^s_T} \le (1/2) \|u-v\|_{\X^s_T}.$$

Now we can prove the existence and uniqueness.
 Consider  (\ref{2.6}) or (\ref{4.0}) in $\X^s_T$. By (\ref{4.1}) and (\ref{4.2}), we know that
$\mathcal{A}$ is a contraction in $ B_M^{\X^s}$, which yields a unique fixed point $u\in B_M^{\X^s}$ for $\mathcal{A}$
using contraction mapping theorem. Thus,  (\ref{2.6}) has a unique solution in $\X^s_T$.

For the critical case with $p = \frac{4-2s}{1-s}$, we  let $r =\frac{2p}{p-2}$ and $q=p$. Note that in (\ref{4.17}), we need that $$
                    \left\|\mathcal{A}[u]\right\|_{\X^s_T} \le C_0 \left(\mu +  M^{p-1} \right) \, .
$$
Thus, instead of (\ref{4.17}) and (\ref{4.19}, for $u$, $v \in B_{2 C_0 \mu}^{\X^s} := \left\{u: \, \|u\|_{\X^s_T} \le 2 C_0 \mu \right\}$, we have
\begin{equation} 
                    \left\|\mathcal{A}[u]\right\|_{\X^s_T} \le {C_0} \left(\mu + C_1\mu^{p-1} \right) \label{4.20}
\end{equation}
and
\begin{equation}
                    \left\|\mathcal{A}[u]-\mathcal{A}[v]\right\|_{\X^0_T} \le \widetilde{C_1} \mu^{p-2} \|u-v\|_{\X^0_T}\, .\label{4.21}
\end{equation}
As long as  $T > 0 $  finite, we can choose $\mu$ small enough so that  ${C_0} \left(\mu + C_1\mu^{p-1} \right) \le 2C_0\mu$ and $\widetilde{C_1} \mu^{p-2} \le (1/2)$. Hence, (\ref{4.1}) and (\ref{4.2}) also hold for this case. This leads to a conclusion that there is a fixed point of (\ref{2.6}) in $B_{2 C_0\mu}^{\X^s}$. By contraction mapping theorem, the existence and uniqueness of this problem is also guaranteed. Here, we note that the initial and boundary conditions are small. However, for the initial and boundary conditions not small, the arguments in Sections 4.5 and 4.7 of \cite{cazenave} can be carried out similarly and are omitted here.

For $ s = 1$, pick $u$ in $B_M^{\X^1}$. Let $(q,r)$ be any admissible pair  with $r > 2$.
\begin{align*}
\left\|\nabla f(u)(t)\right\|_{L^2} &\lesssim \left\|f'(u)(t)\right\|_{L^{\frac{2r}{r-2}}} \|\nabla u(t)\|_{L^r} \lesssim \left\||u(t)|^{p-2}\right\|_{L^{\frac{2r}{r-2}}} \|\nabla u\|_{L^r}\\
 & = \|u(t)\|_{L^{\frac{2r(p-2)}{r-2}}}^{p-2} \|\nabla u(t)\|_{L^r} \lesssim \|u(t)\|_{H^1}^{p-2} \|\nabla u(t)\|_{L^r}
                \end{align*}
by the Sobolev embedding theorem (here, note that we can let $ p > 2$ by choosing $r$ near 2). Thus,
                \begin{equation*}
                    \left\|f(u)(t)\right\|_{H^1} \lesssim  \|u(t)\|_{H^1}^{p-2} \|u(t)\|_{W^{1,r}}\, .
                \end{equation*}
Consequently, adding the norm and applying the Holder's inequality with respect to $t$ show that
                \begin{align}
                      \left\|f(u)\right\|&_{L^1_t  \left([0,T]; H^1_{xy}(\R \times \R^+)\right)} \lesssim \|u\|_{L^{\frac{q(p-2)}{q-1}}_t \left([0,T]; H^1_{xy}(\R \times \R^+)\right)}^{p-2} \cdot \|u\|_{L^q_t \left([0,T]; W^{1,r}_{xy}(\R \times \R^+)\right)} \nonumber \\
                    &\lesssim T^{\frac{q-1}{q}} \|u\|_{L^{\infty}_t \left([0,T]; H^1_{xy}(\R \times \R^+)\right)}^{p-2} \|u\|_{L^q_t \left([0,T]; W^{1,r}_{xy}(\R \times \R^+)\right)}\, .\label{4.22}
                \end{align}
Moreover, Sobolev embedding theorem gives
                \begin{align*}
                \Big \|\nabla & \big [|u|^{p-2}u(t) - |v|^{p-2}v(t)\big ] \Big \|_{L^2} \\
                &= \left\|\int_0^1 \nabla \left[\dfrac{p}{2} |w(t)|^{p-2} + \left(\dfrac{p}{2}-1\right) |w(t)|^{p-4} w^2(t)\right] \cdot (u(t)-v(t)) \, d\theta \right\|_{L^2} \\
                &\le \sup_{\theta\in[0,1]} \left\|\nabla \left[\left(\dfrac{p}{2} |w(t)|^{p-2} + \left(\dfrac{p}{2}-1\right) |w(t)|^{p-4} w^2(t)\right) \cdot (u(t)-v(t))\right] \right\|_{L^2} \\
                &\le \sup_{\theta\in[0,1]} \left\|\nabla \left[\dfrac{p}{2} |w(t)|^{p-2} + \left(\dfrac{p}{2}-1\right) |w(t)|^{p-4} w^2(t)\right] \right\|_{L^{r_1}} \cdot \|u(t)-v(t)\|_{L^{\frac{2r_1}{r_1-2}}} \\
                &\qquad  + \sup_{\theta\in[0,1]} \left\|\dfrac{p}{2} |w(t)|^{p-2} + \left(\dfrac{p}{2}-1\right) |w(t)|^{p-4} w^2(t) \right\|_{L^{\frac{2r}{r-2}}} \cdot \| (u(t)-v(t))\|_{W^{1,r}}
                \end{align*}
where $r_1 > 2 $ is to be chosen later.   Notice that
                \begin{align*}
                 \Big \|\nabla \Big [\frac{p}{2} & |w(t)|^{p-2} + \left(\frac{p}{2}-1\right) |w(t)|^{p-4} w^2(t)\Big ] \Big \|_{L^{r_1}} \\
                &\le \left|\dfrac{p}{2}\right| \cdot \left\|\nabla |w(t)|^{p-2} \right\|_{L^{r_1}} + \left|\dfrac{p}{2}-1\right| \cdot \left\|\nabla |w(t)|^{p-4} w^2(t) \right\|_{L^{r_1}} \\
                &\lesssim \left\||w(t)|^{p-4} |w(t) + \overline{w}(t)| \cdot |\nabla w(t)| \right\|_{L^{r_1}} + \left\||w(t)|^{p-3} \cdot |\nabla w(t)| \right\|_{L^{r_1}} \, .
                \end{align*}
Hence
                \begin{align*}
                 \Big \|\nabla & \big [|u(t)|^{p-2}u(t) - |v(t)|^{p-2}v(t)\big ] \Big \|_{L^2} \\
                &\lesssim \sup_{\theta\in[0,1]} \left\||w(t)|^{p-3} \cdot |\nabla w(t)| \right\|_{L^{r_1}} \cdot \|u(t)-v(t)\|_{L^{\frac{2r_1}{r_1-2}}} + \sup_{\theta\in[0,1]} \|w(t)\|_{L^{\frac{2r(p-2)}{r-2}}}^{p-2} \cdot \| (u(t)-v(t))\|_{W^{1,r}}\, .
                \end{align*}
If $p=3$, let $r_1=r$, which gives
                \begin{align*}
                \Big \|\nabla & \big [|u(t)|^{p-2}u(t) - |v(t)|^{p-2}v(t)\big ] \Big \|_{L^2} \\
                &\lesssim \sup_{\theta\in[0,1]} \| w(t)\|_{W^{1,r}} \cdot \|u(t)-v(t)\|_{W^{1,r}} + \sup_{\theta\in[0,1]} \|\nabla w(t)\|_{L^r} \cdot \|\nabla (u(t)-v(t))\|_{L^r} \\
                &\le \left(\|\nabla u(t)\|_{L^r} + \|\nabla v(t)\|_{L^r}\right) \cdot \| u(t)-v(t)\|_{W^{1,r}}\, .
                \end{align*}
If $p>3$, let $2<r_1<r$ and obtain
                \begin{align*}
               \Big \|\nabla & \big [|u(t)|^{p-2}u(t) - |v(t)|^{p-2}v(t)\big ] \Big \|_{L^2} \\
                &\lesssim \sup_{\theta\in[0,1]} \left\||w(t)|^{p-3} \cdot |\nabla w(t)| \right\|_{L^{r_1}} \cdot \|u(t)-v(t)\|_{L^{\frac{2r_1}{r_1-2}}} + \sup_{\theta\in[0,1]} \|w(t)\|_{L^{\frac{2r(p-2)}{r-2}}}^{p-2} \cdot \| u(t)-v(t)\|_{W^{1,r}} \\
                &\le \sup_{\theta\in[0,1]} \|w(t)\|_{L^{\frac{rr_1(p-3)}{r-r_1}}}^{p-3} \cdot \|\nabla w(t)\|_{L^r} \cdot \|u(t)-v(t)\|_{L^{\frac{2r_1}{r_1-2}}} + \sup_{\theta\in[0,1]} \|w(t)\|_{L^{\frac{2r(p-2)}{r-2}}}^{p-2} \cdot \| u(t)-v(t)\|_{W^{1,r}} \\
                &\le \left(\| u(t)\|_{W^{1,r}}^{p-2} + \| v(t)\|_{W^{1,r}}^{p-2}\right) \cdot \|u(t)-v(t)\|_{W^{1,2}}  + \left(\|u(t)\|_{W^{1,2}}^{p-2} + \| v(t)\|_{W^{1,2}}^{p-2}\right) \cdot \|u(t)-v(t)\|_{W^{1,r}}\, .
                \end{align*}
Therefore,
                \begin{eqnarray*}
                \left\||u(t)|^{p-2}u(t) - |v(t)|^{p-2}v(t)\right\|_{H^1} &\lesssim& \left(\|u(t)\|_{W^{1,r}}^{p-2} + \| v(t)\|_{W^{1,r}}^{p-2}\right) \cdot \|u(t)-v(t)\|_{H^1} \\
                && + \left(\|u(t)\|_{H^1}^{p-2} + \| v(t)\|_{H^1}^{p-2}\right) \cdot \|u(t)-v(t)\|_{W^{1,r}}
                \end{eqnarray*}
Integrating with respect to $t$ implies
                \begin{align}
                     \left\|f(u)-f(v)\right\|&_{L^1_t\left([0,T]; H^1_{xy}(\R \times \R^+)\right)}\nonumber \\
                    &= T^{\frac{q-1}{q}} \left(\|u\|_{L^q_t\left([0,T]; W^{1,r}_{xy}(\R \times \R^+) \right)}^{p-2} + \|v\|_{L^{q}_t\left([0,T]; W^{1,r}_{xy}(\R \times \R^+) \right)}^{p-2} \right) \|u-v\|_{L^{\infty}_t\left([0,T]; H^1_{xy}(\R \times \R^+)\right)}\nonumber \\
                    &+ T^{\frac{q-1}{q}} \left(\|u\|_{L^{\infty}_t\left([0,T]; H^1_{xy}(\R \times \R^+) \right)}^{p-2} + \|v\|_{L^{\infty}_t\left([0,T]; H^1_{xy}(\R \times \R^+) \right)}^{p-2} \right) \|u-v\|_{L^q_t\left([0,T]; L^r_{xy}(\R \times \R^+)\right)} \, .\label{4.23}
                \end{align}
Therefore, by (\ref{4.22}),
                \begin{align}
                   \left\|\mathcal{A}[u]\right\|_{\X^1_T}
                    &\le \|W_b [h] \|_{\X^1_T} + \left\|W_{\R^2}\phi \right\|_{\X^1_T} + \left\|\Phi_f \right\|_{\X^1_T} + \left\|W_b \left[(W_{\R^2}\phi)_{y=0}\right] \right\|_{\X^1_T} + \left\|W_b \left[(\Phi_f)_{y=0}\right] \right\|_{\X^1_T} \nonumber \\
                    &\lesssim \|h\|_{{\cal H}^s (0,T)} + \left\|\phi \right\|_{H^s_{xy}} + \left\|f \right\|_{L^1_t \left([0,T]; H^s_{xy} \right)} \nonumber \\
                    &\le \mu + \left\|f \right\|_{L^1_t \left([0,T]; H^s_{xy} \right)}\
                        {\lesssim} \ \mu + T^{\frac{q-1}{q}} \|u\|_{L^{\infty}_t \left([0,T]; H^1_{xy}(\R \times \R^+)\right)}^{p-2} \|u\|_{L^q_t \left([0,T]; W^{1,r}_{xy}(\R \times \R^+)\right)} \nonumber \\
                    &\le C_0 \left ( \mu + T^{\frac{q-1}{q}} \|u\|_{\X^1_T}^{p-1}\right )\, \label{4.24}
                \end{align}
for some $C_0>0$. Furthermore, by (\ref{4.23}),
                \begin{align*}
                    & \left\|\mathcal{A}[u]-\mathcal{A}[v]\right\|_{\X^1_T} \le \left\|\Phi_f(u)-\Phi_f(v) \right\|_{\X^1_T} + \left\|W_b \left[(\Phi_f(u)-\Phi_f(v))_{y=0}\right] \right\|_{\X^1_T} \\
                    &\qquad \lesssim  \left\|f(u)-f(v) \right\|_{L^1_t \left([0,T]; L^2_{xy} \right)}  \le C_1 T^{\frac{q-1}{q}} M^{p-2} \|u-v\|_{\X^1_T}\, .
                \end{align*}
In order to apply the contraction mapping theorem, we need
                \begin{equation}
                    C_1 T^{\frac{q-1}{q}} M^{p-2} \le 1/2\, ,\label{4.25}
                \end{equation}
and
                \begin{equation}
                    0 < T \le \left(\frac{M-C_0\mu}{C_0 M^{p-1}}\right)^{\frac{q}{q-1}}\, .\label{4.26}
                \end{equation}
Therefore, by above estimates and the contraction mapping theorem, we show that there is a unique solution to (\ref{2.6}) in $\X^1_T$.

The case for $ s > 1$ has been studied for the pure initial value problem (\ref{1.2}), whose well-posedness was
obtained in $\mathcal{Z}^s$ with $p\ge s+1$ if $s\in\Z$ or $p\ge[s]+2$ if $s\notin\Z$ only when $p$ is not an even integer (see Section 4.10 in \cite{cazenave}, where only integer derivatives are considered, which can be easily extended to fractional derivatives using chain rules and product rules for fractional derivatives). For the IBVP (\ref{2.6}), the only different part is the estimates for the boundary integral operator. However, for $ s > 1$, Propositions \ref{prop3.2}, \ref{prop3.4} and \ref{prop3.5} with
$q = \infty, r =2 $ imply that
                \begin{align*}
                    &\|W_b [h]\|_{\mathcal{Z}^s_T} + \left\|W_b \left[(W_{\R^2}\phi)_{y=0}\right] \right\|_{\mathcal{Z}^s_T} + \left\|W_b \left[(\Phi_f)_{y=0}\right] \right\|_{\mathcal{Z}^s_T} \\
                    &\lesssim \|h\|_{{\cal H}^s (0,T)} + \left\|\phi \right\|_{H^s_{xy}} + \left\|f \right\|_{L^1_t \left([0,T]; H^s_{xy} \right)} \, ,
                \end{align*}
and
                \begin{align*}
                    & \left\|W_b \left[(\Phi_f(u)-\Phi_f(v))\big |_{y=0}\right] \right\|_{\mathcal{Z}^0_T} \lesssim \left\|f(u)-f(v) \right\|_{L^1_t \left([0,T]; L^2_{xy} \right)} \, .
                \end{align*}
Hence, the proof in Section 4.10 of \cite{cazenave} can be applied and we can obtain that for $p\ge s+1$ if $s\in\Z$ or $p\ge[s]+2$ if $s\notin\Z$, (\ref{4.3}) and (\ref{4.4}) holds. Thus, the contraction mapping principle yields a fixed point $u \in \mathcal{Z}^0_T$, which then gives
a unique solution of (\ref{2.6}) in $\mathcal{Z}^s_T$ for $s>1$ since $\mathcal{Z}^s_T$ is a reflexive Banach space.

Here, for $p$ an even integer, $f\in C^{\infty}(\C)$, i.e., there is a unique solution $u \in C_t \left([0,T]; \, H^s_{xy}(\R \times \R^+)\right)$ for any $s>1$ if the initial and boundary conditions are in appropriate spaces discussed above. However, for general $s > 1$ we can only allow $s\le p-1$ if $s\in\Z$ or $[s] \le p-2$ if $s\notin\Z$ so that the existence and uniqueness hold for $u \in C_t \left([0,T]; \, H^s_{xy}(\R \times \R^+)\right)$.
Finally, Definition \ref{defi1.2} guarantees the existence and uniqueness for problem (\ref{2.1}).
    \end{proof}

Now, we discuss the maximum existence interval of solutions found in Theorem \ref{theo4.3}.
\begin{prop}\label{prop4.4}
        Assume that a unique solution $u$ to the problem (\ref{2.1}) exists in $(a)$ $\X^s_T$ if $0 \le s < 1$ and $3 \le p < \frac{4-2s}{1-s}$, or $s=1$ and $3 \le p < \infty$, or $(b)$ $\mathcal{Z}^s_T$ if $p\ge s+1$ for $s\in\Z$ or $p\ge[s]+2$ for $s\notin\Z$ or when $p$ is an even integer,  for $t \in [ 0, T] $ with $\varphi \in H^s (\R \times \R^+)$ and $h \in {\cal H}^s (0,T_0)$ with any $T_0 > 0$. Let $T_{\max}= \sup T$ and suppose $T_{\max}<\infty$. Also, define $u^{\ast}$ on $[0,T_{\max})$ as the solution of (\ref{2.1}) in $\X^s_{T_{\max}}$, with $u^{\ast}(t) = u(t)$ whose existence and uniqueness have been proved in \emph{Proposition~\ref{theo4.3}}. Then $\displaystyle \lim_{t\uparrow T_{\max}} \|u(t)\|_{H^s(\R \times \R^+)} = \infty$.
    \end{prop}
\begin{proof}
The proof can be easily obtained using classical extension procedure and is omitted. Here, we note that $h$ has to be defined for any large time $t > 0$. Moreover,
we are not able to construct the blow-up statement with the critical case $p = \frac{4-2s}{1-s}$ with $0 \leq s < 1$ because the existence and uniqueness results are only for small initial and boundary data.
\end{proof}

The following continuous dependence of solutions on the initial and boundary data can be proved as well.
\begin{prop}\label{prop4.5} 
Let $0 \le s < p-1$ for $s\in\Z$ or $0 \le [s] \le p-2$ for $s\notin\Z$ ($ 0 \leq s < \infty$ for $p$ even) and $\{\varphi_n\}$ be a sequence of functions in $H^s(\R \times \R^+)$ and $\varphi \in H^s(\R \times \R^+)$ so that $\varphi_n \to \varphi$ as $n \to \infty$ in $H^s(\R \times \R^+)$. Let $ h $ be a function and $\{h_n\}$ be a sequence of functions
such that
$$
h , h_n \in {\cal H}^s (0,T) \quad\mbox{and}\quad  h_n \to h \quad \mbox {as\ } n \to \infty
$$
in ${\cal H}^s (0,T)$ for any $T>0$. Let $u_n$ be the solutions to the equation (\ref{2.6}) with $u_n(x,y,0)=\varphi_n(x,y)$ and  $u_n(x,0,t)=h_n(x,t)$ and $u$ be the solution with $u(x,y,0)=\varphi(x,y)$ and $u(x,0,t)=h(x,t)$, respectively. Then $u_n \to u$ as $n \to \infty$ in $X_T$ with $\left\|u_n\right\|_{X_T}\le M$ where $X_T=\X^s_T$ for $0 \leq s < 1$ and $3 \leq p < \frac{4-2s}{1-s}$ or $\X^1_T$ for $s =1$ or $\mathcal{Z}^s_T$ for $s > 1$, respectively.
\end{prop}
\begin{proof}
First, we consider $0 \le s < 1$ and $3 \le p < \frac{4-2s}{1-s}$. \emph{Proposition~\ref{theo4.3}} and \emph{Proposition~\ref{prop4.4}} guarantee the existence of a common existence interval $[0, T_c]$ for $u_n$ and $u$  because of the choice of $T_{\max}$ only dependent upon the bounds of initial and boundary data in their respective spaces. Furthermore, from the proof of (\ref{4.1}), we can obtain
            \begin{align*}
                & \left\|u-u_n\right\|_{\X^s_T} = \left\|\mathcal{A}[u]-\mathcal{A}[u_n]\right\|_{\X^s_T} \le C \bigg ( \left\|\varphi-\varphi_n\right\|_{H^s(\R \times \R^+)} \\
                & \qquad + \left\|h-h_n\right\|_{{\cal H}^s (0,T)}  + T^{1 - \frac{(1-s)(p-2)}{2}} M^{p-2} \|u-u_n\|_{\X^s_T}\bigg ) \, .
            \end{align*}
Let $T$ be sufficiently small so that $T^{1 - \frac{(1-s)(p-2)}{2}} M^{p-2} <1/2$. Then
            \begin{equation*}
                \left \|u-u_n\right\|_{\X^s_T} \le 2 C \bigg ( \left\|\varphi-\varphi_n\right\|_{H^s(\R \times \R^+)} +
                 \left\|h-h_n\right\|_{{\cal H}^s (0,T)}  \bigg ) \, .
            \end{equation*}
Since $T$ only depends on the uniform bounds for $u, u_n, \varphi, \varphi_n , h, h_n$ in their respective norms for $t \in [ 0, T_c]$, the above inequality holds for $T, 2T, \dots$ until reaching $T_c$. The continuous dependence is proved for $0 \leq s < 1$.

The case for $s =1$ can be obtained similarly. For $s > 1$, the proof of continuous dependence in $\mathcal{Z}^s_T$ has been given for the purely IVP (\ref{1.2}) (see Section 4.10 in \cite{cazenave}), which can be easily extended to the IBVP (\ref{1.1}).  Thus, the proof of the continuous dependence is completed.
\end{proof}

Hence, Theorem \ref{theo1.3} is proved. Using Theorem \ref{theo1.3}, the proof of Theorem \ref{theo1.4} follows
directly from the argument given in \cite{Sun_CUC_Evo}. Here, we note that the regularity theorems discussed in Sections 5.1-5.4 of \cite{cazenave}
hold for the IBVP (\ref{1.1}), i.e., the persistence of regularity defined in \cite{Sun_CUC_Evo} is valid and unconditional
well-posedness theorems in \cite{Sun_CUC_Evo} can be applied for (\ref{1.1}) (also see Section 4 of \cite{Sun_CUC_Evo}).

\section{Global Well-posedness}\label{sect5}

    In this section, we will discuss the global existence of solutions for (\ref{1.1}) with any $T \in (0,\infty]$.
To prove the global existence, we first derive several identities.
\begin{lemma} \label{lemm5.1}
        If the solution $u$ of (\ref{1.1}) exists for any $ t > 0 $ and is smooth enough, then $u$ satisfies the following identities.
\begin{align}
& (|u|^2)_t = -2 \text{Im} \left[(u_x \overline{u})_x + (u_y \overline{u})_y\right] \, ,\label{5.1}\\
&\left(|u_x|^2+|u_y|^2-\dfrac{2\lambda}{p}|u|^p\right)_t = 2 \text{Re} \left[(\overline{u_t} u_x)_x + (\overline{u_t} u_y)_y\right] \, ,\label{5.2}\\
& \left(|u_y|^2-|u_x|^2+\dfrac{2\lambda}{p}|u|^p\right)_y = -2 \text{Re}(\overline{u}_x u_y)_x - i(u \overline{u_y})_t + i(u \overline{u_t})_y\, . \label{5.3}
\end{align}
\end{lemma}
\begin{proof}
The first identity (\ref{5.1}) can be obtained by multiplying both sides of the equation (\ref{2.1}) by $\overline{u_t}$ and then retaining the imaginary parts.
For the proof of (\ref{5.2}), first we know that
            \begin{eqnarray*}
                (|u_x|^2)_t &=& (u_x \overline{u}_x)_t = 2 \text{Re} (u_{xt} \overline{u}_x) = 2 \text{Re} (\overline{u_t} u_x)_x - 2 \text{Re} (\overline{u_t} u_{xx}) \, ,\\
                (|u_y|^2)_t &=& 2 \text{Re} (\overline{u_t} u_y)_y - 2 \text{Re} (\overline{u_t} u_{yy})\, .
            \end{eqnarray*}
Add them together to have
            \begin{eqnarray*}
                (|u_x|^2+|u_y|^2)_t &=& 2 \text{Re} \left[(\overline{u_t} u_x)_x+(\overline{u_t} u_y)_y\right] - 2 \text{Re} (\overline{u_t} \Delta u) \\
                &=& 2 \text{Re} \left[(\overline{u_t} u_x)_x+(\overline{u_t} u_y)_y\right] + 2 \text{Re} \left(i|u_t|^2 + \lambda\overline{u_t} u |u|^{p-2}\right) \\
                &=& 2 \text{Re} \left[(\overline{u_t} u_x)_x+(\overline{u_t} u_y)_y\right] + \dfrac{2\lambda}{p} \cdot \dfrac{p}{2} \cdot \left(|u|^2\right)^{\frac{p}{2}-1}\left(|u|^2\right)_t \\
                &=& 2 \text{Re} \left[(\overline{u_t} u_x)_x+(\overline{u_t} u_y)_y\right] + \left(\dfrac{2\lambda}{p} |u|^p\right)_t\, ,
            \end{eqnarray*}
which gives (\ref{5.2}), where (\ref{2.1}) has been used. The proof of (\ref{5.3}) is as follows.
            \begin{eqnarray*}
                 -\left(|u_y|^2-|u_x|^2\right)_y & =& (u_y \overline{u}_x)_x + (\overline{u_y} u_x)_x - u_y(\overline{u}_{xx}+\overline{u}_{yy}) - \overline{u_y}(u_{xx}+u_{yy}) \\
                &=& 2\text{Re}(u_y \overline{u}_x)_x + u_y(-i\overline{u_t}+\lambda \overline{u}|u|^{p-2}) + \overline{u_y}(i u_t+\lambda u|u|^{p-2}) \\
                &=& 2\text{Re}(u_y \overline{u}_x)_x + i(u_t \overline{u_y} - \overline{u_t} u_y) + \lambda |u|^{p-2} (u_y \overline{u} + \overline{u_y} u) \\
                &=& 2\text{Re}(u_y \overline{u}_x)_x + i\left[(u \overline{u_y})_t - (\overline{u_t} u)_y\right] + \dfrac{2\lambda}{p} \cdot \dfrac{p}{2} \cdot \left(|u|^2\right)^{\frac{p}{2}-1}\left(|u|^2\right)_y \\
                &=& 2\text{Re}(u_y \overline{u}_x)_x + i(u \overline{u_y})_t - i(\overline{u_t} u)_y + \left(\dfrac{2\lambda}{p}|u|^p\right)_y\, .
            \end{eqnarray*}
\end{proof}

Next, we derive an {\it a-priori} estimate of the solution for (\ref{2.1}) in $H^1(\R\times\R^+)$.
\begin{prop}\label{prop5.2} 
        Assume that either $3 \le p < \infty$ if $\lambda<0$ or $3 \le p \le \frac{10}{3}$ if $\lambda>0$. Let $T>0$ be given. If $u$ is a solution of (\ref{2.1}) in $C_t \left([0,T]; \, H^1_{xy}(\R \times \R^+)\right)$, then there exists a nondecreasing function $\phi: \, \R^+ \to \R^+$  such that for $\varphi \in H^1(\R\times\R^+)$ and $h \in H^1(\R\times\R^+)$,
        \begin{equation}
            \sup_{t\in[0,T]} \|u(t)\|_{H^1(\R\times\R^+)} \le \phi\left (\|\varphi\|_{H^1(\R\times\R^+)}+\|h\|_{H^1(\R\times\R^+)}\right )\, . \label{5.4}
        \end{equation}
\end{prop}
\begin{proof}
First, we show the $L^2$-norm estimate of the solution over variables $x$ and $y$. By (\ref{5.1}),
\begin{align*}
                 \|u(t)\|_{L^2_{xy}}^2 &= \int_{-\infty}^{\infty} \int_0^{\infty} |u(x,y,t)|^2 \, dy \, dx \\
                &= \int_{-\infty}^{\infty} \int_0^{\infty} |u(x,y,0)|^2 \, dy \, dx + \int_{-\infty}^{\infty} \int_0^{\infty} \int_0^t \left(|u(x,y,\tau)|^2\right)_t \, d\tau \, dy \, dx \\
                &{=} \|\varphi\|_{L^2_{xy}}^2 - 2 \text{Im} \int_{-\infty}^{\infty} \int_0^{\infty} \int_0^t \left[(u_x(x,y,\tau) \overline{u}(x,y,\tau))_x + (u_y(x,y,\tau) \overline{u}(x,y,\tau))_y\right] \, d\tau \, dy \, dx \\
                &= \|\varphi\|_{L^2_{xy}}^2 + \int_{-\infty}^{\infty} \int_0^t 2 \text{Im} \left(u_y(x,0,\tau) \overline{u}(x,0,\tau)\right) \, d\tau \, dx\, ,
            \end{align*}
which gives
            \begin{equation}
                \|u(t)\|_{L^2_{xy}}^2 \le \|\varphi\|_{L^2_{xy}}^2 + 2 \|h\|_{L^2_{xt}} \left(\int_{-\infty}^{\infty} \int_0^t \left|u_y(x,0,\tau)\right|^2 \, d\tau \, dx\right)^{\frac{1}{2}}\, .\label{5.5}
            \end{equation}
Then, we derive the estimate for $u_y(x,0,\tau)$. Integrate (\ref{5.3}) with respect to $x$, $y$ and $t$ to obtain
            \begin{align*}
                \int_{-\infty}^{\infty} & \int_0^t \left|u_y(x,0,\tau)\right|^2 \, d\tau \, dx - \int_{-\infty}^{\infty} \int_0^t \left|u_x(x,0,\tau)\right|^2 \, d\tau \, dx \\
                &= i \int_{-\infty}^{\infty} \int_0^{\infty} \left|u(x,y,t) \overline{u_y}(x,y,t)\right| \, dy \, dx -i \int_{-\infty}^{\infty} \int_0^{\infty} \left|u(x,y,0) \overline{u_y}(x,y,0)\right| \, dy \, dx \\
                & \qquad + i \int_{-\infty}^{\infty} \int_0^t \left|u(x,0,\tau) \overline{u_t}(x,0,\tau)\right| \, d\tau \, dx - \dfrac{2\lambda}{p} \int_{-\infty}^{\infty} \int_0^t \left|u(x,0,\tau)\right|^p \, d\tau \, dx\, ,
            \end{align*}
which implies
            \begin{align*}
            \int_{-\infty}^{\infty} \int_0^t & \left|u_y(x,0,\tau)\right|^2 \, d\tau \, dx = \int_{-\infty}^{\infty} \int_0^t \left|u_x(x,0,\tau)\right|^2 \, d\tau \, dx \\
                &= i \int_{-\infty}^{\infty} \int_0^{\infty} \left|u(x,y,t) \overline{u_y}(x,y,t)\right| \, dy \, dx -i \int_{-\infty}^{\infty} \int_0^{\infty} \left|u(x,y,0) \overline{u_y}(x,y,0)\right| \, dy \, dx \\
                & \qquad + i \int_{-\infty}^{\infty} \int_0^t \left|u(x,0,\tau) \overline{u_t}(x,0,\tau)\right| \, d\tau \, dx - \dfrac{2\lambda}{p} \int_{-\infty}^{\infty} \int_0^t \left|u(x,0,\tau)\right|^p \, d\tau \, dx \\
                &\le \|h_x\|_{L^2_{xt}}^2 + \left(\int_{-\infty}^{\infty} \int_0^{\infty} |u(x,y,t)|^2 \, dy \, dx\right)^{\frac{1}{2}} \cdot \left(\int_{-\infty}^{\infty} \int_0^{\infty} |u_y(x,y,t)|^2 \, dy \, dx\right)^{\frac{1}{2}} \\
                &\qquad + \|\varphi\|_{L^2_{xy}} \cdot \|\varphi_y\|_{L^2_{xy}} + \|h\|_{L^2_{xt}} \cdot \|h_t\|_{L^2_{xt}} + \dfrac{2|\lambda|}{p} \|h\|_{L^p_{xt}}^p \\
                &= \|h_x\|_{L^2_{xt}}^2 + \|u\|_{L^2_{xy}} \cdot \|u_y\|_{L^2_{xy}} + \|\varphi\|_{L^2_{xy}} \cdot \|\varphi_y\|_{L^2_{xy}} + \|h\|_{L^2_{xt}} \cdot \|h_t\|_{L^2_{xt}} + \dfrac{2|\lambda|}{p} \|h\|_{L^p_{xt}}^p\, .
            \end{align*}
It is derived from (\ref{5.5}) that
            \begin{align*}
                 \|u(t)&\|_{L^2_{xy}}^2 \le  \|\varphi\|_{L^2_{xy}}^2 + 2 \|h\|_{L^2_{xt}} \left(\int_{-\infty}^{\infty} \int_0^t \left|u_y(x,0,\tau)\right|^2 \, d\tau \, dx\right)^{\frac{1}{2}} \\
                \le& 2 \|h\|_{L^2_{xt}} \cdot \left(\|u\|_{L^2_{xy}} \|u_y\|_{L^2_{xy}} + \|\varphi\|_{L^2_{xy}} \|\varphi_y\|_{L^2_{xy}} + \|h\|_{L^2_{xt}} \|h_t\|_{L^2_{xt}} + \dfrac{2|\lambda|}{p} \|h\|_{L^p_{xt}}^p + \|h_x\|_{L^2_{xt}}^2\right)^{\frac{1}{2}}  + \|\varphi\|_{L^2_{xy}}^2 \\
                \lesssim& \|u(t)\|_{L^2_{xy}}^{\frac{1}{2}} \left(2 \|h\|_{L^2_{xt}} \|u_y\|_{L^2_{xy}}^{\frac{1}{2}}\right) \\
                & + 2 \|h\|_{L^2_{xt}} \left(\|\varphi\|_{L^2_{xy}}^{\frac{1}{2}} \|\varphi_y\|_{L^2_{xy}}^{\frac{1}{2}} + \|h\|_{L^2_{xt}}^{\frac{1}{2}} \|h_t\|_{L^2_{xt}}^{\frac{1}{2}} + \sqrt{\dfrac{2|\lambda|}{p}} \|h\|_{L^p_{xt}}^{\frac{p}{2}} + \|h_x\|_{L^2_{xt}}\right) + \|\varphi\|_{L^2_{xy}}^2 \\
                \le& \dfrac{1}{4} \|u(t)\|_{L^2_{xy}}^2 + \dfrac{3}{4} \left(2 \|h\|_{L^2_{xt}}\right)^{\frac{4}{3}} \|u_y\|_{L^2_{xy}}^{\frac{2}{3}} \\
                & + 2 \|h\|_{L^2_{xt}} \left(\|\varphi\|_{L^2_{xy}}^{\frac{1}{2}} \|\varphi_y\|_{L^2_{xy}}^{\frac{1}{2}} + \|h\|_{L^2_{xt}}^{\frac{1}{2}} \|h_t\|_{L^2_{xt}}^{\frac{1}{2}} + \sqrt{\dfrac{2|\lambda|}{p}} \|h\|_{L^p_{xt}}^{\frac{p}{2}} + \|h_x\|_{L^2_{xt}}\right) + \|\varphi\|_{L^2_{xy}}^2\, ,
            \end{align*}
which yields
\begin{align}  
                \|u(t)\|_{L^2_{xy}}^2 \le &  \|u_y\|_{L^2_{xy}}^{\frac{2}{3}} \left(2 \|h\|_{L^2_{xt}}\right)^{\frac{4}{3}} \nonumber\\
                & + \dfrac{4}{3} \|h\|_{L^2_{xt}} \left(\|\varphi\|_{L^2_{xy}}^{\frac{1}{2}} \|\varphi_y\|_{L^2_{xy}}^{\frac{1}{2}} + \|h\|_{L^2_{xt}}^{\frac{1}{2}} \|h_t\|_{L^2_{xt}}^{\frac{1}{2}} + \sqrt{\dfrac{2|\lambda|}{p}} \|h\|_{L^p_{xt}}^{\frac{p}{2}} + \|h_x\|_{L^2_{xt}}\right) + \dfrac{4}{3} \|\varphi\|_{L^2_{xy}}^2  \, .\label{5.6}
            \end{align}
With $L^{\infty}L^2$-norm of $u$, we then integrate the identity (\ref{5.2}) with respect to $x$, $y$ and $t$. For the sake of simplicity of notations, we  let $C = C\left(\|\varphi\|_{H^1_{xy}}\right)$ be an increasing function depending on the initial data and  $C (t)= C\left(\|h\|_{H^1_{xt}}\right)$ denote an increasing function for the boundary value and the time $t$. Moreover, $C(0)=0$ and they both equal zero for $\| \varphi\|_{H^1_{xy}}=0$ and $\|h\|_{H^1_{xt}}=0$.

First, assume $\lambda<0$. Apply (\ref{5.2}) and (\ref{5.6}) to obtain
            \begin{align*}
                 & \int_{-\infty}^{\infty}  \int_0^{\infty} \left(|u_y(t)|^2+|u_x(t)|^2\right) \, dy \, dx = \int_{-\infty}^{\infty} \int_0^{\infty} \left(|u_y(x,y,0)|^2+|u_x(x,y,0)|^2\right) \, dy \, dx \\
                & \quad + \int_{-\infty}^{\infty} \int_0^{\infty} \dfrac{2\lambda}{p} |u|^p \, dy \, dx - \int_{-\infty}^{\infty} \int_0^{\infty} \dfrac{2\lambda}{p} |u(x,y,0)|^p \, dy \, dx - 2\text{Re} \int_{-\infty}^{\infty} \int_0^t \overline{u_t}(x,0,\tau) u_y(x,0,\tau) \, d\tau \, dx \\
                &\le \|\varphi\|_{H^1_{xy}}^2 + \dfrac{2|\lambda|}{p} \|\varphi\|_{L^p_{xy}}^p + \int_{-\infty}^{\infty} \int_0^t \left|u_y(x,0,\tau)\right|^2 \, d\tau \, dx + \|h_t\|_{L^2_{xy}}^2 \\
                &\le \|u_y\|_{L^2_{xy}} \cdot \|u\|_{L^2_{xy}} + \left(\|\varphi\|_{L^2_{xy}} \cdot \|\varphi_y\|_{L^2_{xy}} + \|h\|_{L^2_{xt}} \cdot \|h_t\|_{L^2_{xt}} + \dfrac{2|\lambda|}{p} \|h\|_{L^p_{xt}}^p + \|h_x\|_{L^2_{xt}}^2 \right. \\
                & \quad \left. + \|\varphi\|_{H^1_{xy}}^2 + \dfrac{2|\lambda|}{p} \|\varphi\|_{L^p_{xy}}^p + \|h_t\|_{L^2_{xy}}^2 \right) \\
                &\le \dfrac{1}{2} \|u_y\|_{L^2_{xy}}^{\frac{4}{3}} \left(2 \|h\|_{L^2_{xt}}\right)^{\frac{4}{3}} + \|u_y\|_{L^2_{xy}}\left(C+C(t)\right) \\
                & \quad + \left(\|\varphi\|_{L^2_{xy}} \cdot \|\varphi_y\|_{L^2_{xy}} + \|h\|_{L^2_{xt}} \cdot \|h_t\|_{L^2_{xt}} + \dfrac{2|\lambda|}{p} \|h\|_{L^p_{xt}}^p + \|h_x\|_{L^2_{xt}}^2 + \|\varphi\|_{H^1_{xy}}^2 + \dfrac{2|\lambda|}{p} \|\varphi\|_{L^p_{xy}}^p + \|h_t\|_{L^2_{xy}}^2 \right) \\
                &= \dfrac{1}{2} \|u_y\|_{L^2_{xy}}^{\frac{4}{3}} \left(2 \|h\|_{L^2_{xt}}\right)^{\frac{4}{3}} + \|u_y\|_{L^2_{xy}}\left(C+C(t)\right) + C + C(t) \le \dfrac{1}{6} \|u_y\|_{L^2_{xy}}^2 + C + C(t)\, .
            \end{align*}
By moving the terms of $\|u_y\|_{L^2_{xy}}^2$ on the right side to the left side, we conclude that $\|u(t)\|_{H^1(\R\times\R^+)}$ is uniformly bounded by the initial and boundary data, i.e.,
            \begin{equation*}
                 \|u(t)\|_{H^1(\R\times\R^+)} \le C + C(t) := \phi_1\, .
            \end{equation*}
One may notice that during the derivation of above inequality with $C$ and $C(t)$, there are some terms involving $\|\varphi\|_{L^p_{xy}}$ and $\|h\|_{L^p_{xt}}$, which can be bounded by their $H^1$-norms using Sobolev embedding
theorem since $H^1$ is embedded in $L^p$ for any $2 \leq p < \infty $ if the domain is $2$-dimensional.

Next, let us consider the case for $\lambda>0$. Here, we need the following Gagliardo-Nirenberg's inequality,
            \begin{equation*}
                \|u(t)\|_{L^p_{xy}} \lesssim \left(\|u_x(t)\|_{L^2_{xy}} + \|u_y(t)\|_{L^2_{xy}}\right)^{1-\frac{2}{p}} \cdot \|u(t)\|_{L^2_{xy}}^{\frac{2}{p}}\, ,
            \end{equation*}
i.e.,
            \begin{equation}
                \int_{\R^2} |u(t)|^p \, dy \, dx \lesssim \left(\|u_x(t)\|_{L^2_{xy}} + \|u_y(t)\|_{L^2_{xy}}\right)^{p-2} \cdot \|u(t)\|_{L^2_{xy}}^2\, .\label{5.7}
            \end{equation}

If no confusion arises, again we denote ``$\lesssim$" by ``$\le$" in the proof.
From (\ref{5.2}) and using the same functions $C$ and $C(t)$ together with (\ref{5.6}) and (\ref{5.7}) and the estimate obtained for the $L^2$-norm of $u_y ( x, 0, t)$, it is shown that
            \begin{align*}
                  \int_{-\infty}^{\infty}& \int_0^{\infty} \left(|u_y(t)|^2+|u_x(t)|^2\right) \, dy \, dx = \int_{-\infty}^{\infty} \int_0^{\infty} \left(|u_y(x,y,0)|^2+|u_x(x,y,0)|^2\right) \, dy \, dx \\
                & + \int_{-\infty}^{\infty} \int_0^{\infty} \dfrac{2\lambda}{p} |u|^p \, dy \, dx - \int_{-\infty}^{\infty} \int_0^{\infty} \dfrac{2\lambda}{p} |u(x,y,0)|^p \, dy \, dx - 2\text{Re} \int_{-\infty}^{\infty} \int_0^t \overline{u_t}(x,0,\tau) u_y(x,0,\tau) \, d\tau \, dx \\
                \le& \|\varphi\|_{H^1_{xy}}^2 + \dfrac{2|\lambda|}{p} \|\varphi\|_{L^p_{xy}}^p + \int_{-\infty}^{\infty} \int_0^t \left|u_y(x,0,\tau)\right|^2 \, d\tau \, dx + \|h_t\|_{L^2_{xy}}^2 + \dfrac{2\lambda}{p} \int_{-\infty}^{\infty} \int_0^{\infty} |u|^p \, dy \, dx \\
                \le& \dfrac{2\lambda}{p} \int_{-\infty}^{\infty} \int_0^{\infty} |u|^p \, dy \, dx + \dfrac{1}{6} \|u_y\|_{L^2_{xy}}^2 + C + C(t) \\
                {\le}& \left(\|u_x\|_{L^2_{xy}} + \|u_y\|_{L^2_{xy}}\right)^{p-2} \cdot \|u(t)\|_{L^2_{xy}}^2 + \dfrac{1}{6} \|u_y\|_{L^2_{xy}}^2 + C + C(t) \\
                {\le}& \left(\|u_x\|_{L^2_{xy}} + \|u_y\|_{L^2_{xy}}\right)^{p-2} \cdot \left(C(t) \|u_y\|_{L^2_{xy}}^{\frac{2}{3}} + C(t) + C\right) + \dfrac{1}{6} \|u_y\|_{L^2_{xy}}^2 + C + C(t) \\
                \le& C(t) \|u_y\|_{L^2_{xy}}^{\frac{2}{3}} \|u_x\|_{L^2_{xy}}^{p-2} + C(t) \|u_y\|_{L^2_{xy}}^{\frac{3p-4}{3}} + \left(\|u_x\|_{L^2_{xy}} + \|u_y\|_{L^2_{xy}}\right)^{p-2} \cdot (C(t) + C) \\
                & \qquad + \dfrac{1}{6} \|u_y\|_{L^2_{xy}}^2 + C + C(t) \\
                \le& \dfrac{1}{6} \|u_y\|_{L^2_{xy}}^2 + C(t) \|u_x\|_{L^2_{xy}}^{{\frac{3(p-2)}{2}}} + C(t) \|u_y\|_{L^2_{xy}}^{\frac{3p-4}{3}} + \dfrac{1}{2} \|u_x\|_{L^2_{xy}}^2 + \dfrac{1}{6} \|u_y\|_{L^2_{xy}}^2
                 + \dfrac{1}{6} \|u_y\|_{L^2_{xy}}^2 + C + C(t)\, ,
            \end{align*}
from which we can see that
            \begin{equation*}
                 \|u(t)\|_{H^1(\R\times\R^+)} \le C + C(t) := \phi_2
            \end{equation*}
if and only if $\frac{3(p-2)}{2} < 2$ and $\frac{3p-4}{3} < 2$, i.e., $3 \le p < \frac{10}{3}$.
Therefore, in this case, $ \sum_{|\alpha|=1} \left\|D^{\alpha}u(t)\right\|_{L^2}$ or $\|u(t)\|_{H^1_{xy}}$ is uniformly bounded by $C+C(t)$.

If $p=\frac{10}{3}$, we have $\frac{3(p-2)}{2} = \frac{3p-4}{3} = 2$. Thus,
\begin{align*}
                \int_{-\infty}^{\infty}&  \int_0^{\infty} \left(|u_y(t)|^2+|u_x(t)|^2\right) \, dy \, dx \\
                &\le \left(2 \|h\|_{L^2_{xt}}\right)^{\frac{4}{3}} \|u_y\|_{L^2_{xy}}^{\frac{2}{3}} \|u_x\|_{L^2_{xy}}^{\frac{4}{3}} + \left(2 \|h\|_{L^2_{xt}}\right)^{\frac{4}{3}} \|u_y\|_{L^2_{xy}}^2 + \left(\|u_x\|_{L^2_{xy}} + \|u_y\|_{L^2_{xy}}\right)^{\frac{4}{3}} \cdot (C(t) + C) \\
                & \qquad + \frac{1}{6} \|u_y\|_{L^2_{xy}}^2 + C + C(t) \\
                &\le \frac{1}{6} \|u_y\|_{L^2_{xy}}^2 + c \|h\|_{L^2_{xt}}^2 \|u_x\|_{L^2_{xy}}^2 + \left(2 \|h\|_{L^2_{xt}}\right)^{\frac{4}{3}} \|u_y\|_{L^2_{xy}}^2 + \frac{1}{2} \|u_x\|_{L^2_{xy}}^2 + \frac{1}{6} \|u_y\|_{L^2_{xy}}^2 \\
                & \qquad + \frac{1}{6} \|u_y\|_{L^2_{xy}}^2 + C + C(t) \\
                &\le \frac{1}{2} \|u_y\|_{L^2_{xy}}^2 + \left(c \|h\|_{L^2_{xt}}^2 + \frac{1}{2}\right) \|h\|_{L^2_{xt}}^2 \|u_x\|_{L^2_{xy}}^2 + c \|h\|_{L^2_{xt}}^2 \|u_y\|_{L^2_{xy}}^2 + C + C(t) \\
                &= \left(c \|h\|_{L^2_{xt}}^2 + \frac{1}{2}\right) \left(\|u_y\|_{L^2_{xy}}^2 + \|u_x\|_{L^2_{xy}}^2\right) + C + C(t)\, ,
            \end{align*}
where $c$ is a fixed positive constant. The above inequality is equivalent to
            \begin{equation*}
                \left(\frac{1}{2} - c \|h\|_{L^2_{xt}}^2\right) \left(\|u_y\|_{L^2_{xy}}^2 + \|u_x\|_{L^2_{xy}}^2\right) \le C + C(t)\, .
            \end{equation*}
Since it is assumed that $\|h\|_{L^2_{xt}(\R \times [0,T])} < \infty$, we can partition $[0,T]$ into a finite number of subintervals $(t_{j-1},t_j)$ for $j=1$, $\cdots$, $m$ with $\sup_j |t_j-t_{j-1}| \le \delta$ such that $\|h\|_{L^2_{xt}(\R \times [t_{j-1},t_j])} < (1/4c)$. Starting with $[0,t_1]$, we can move forward over one subinterval $(t_{j-1},t_j)$ at a time to obtain a uniform bound for $\| u \|_{H^1}$ and then use $u(t_j)$ as the initial value for the solution on $(t_j,t_{j+1})$ to have a uniform bound for $\| u \|_{H^1}$. By repeating the process until reaching $T$, we prove the uniform bound of $\| u \|_{H^1}$ for $ t \in [0, T]$.
\end{proof}

Finally, the following global well-posedness of the IBVP (\ref{1.1}) can be obtained from Theorem \ref{theo4.3} and {Propositions~\ref{prop4.5}} and {\ref{prop5.2}}.
\begin{thm}\label{theo5.3}
Assume that either $3 \le p < \infty$ if $\lambda<0$ or $3 \le p \le \frac{10}{3}$ if $\lambda>0$. Then,
        (\ref{1.1}) is globally well-posed  in $H^1(\R \times \R^+)$ if $\varphi\in H^1(\R \times \R^+)$ and $h\in H^1_{t,\text{loc}} \left(\R; \, L^2_x(\R)\right)$ $\bigcap L^2_{t, \text{loc}} \left(\R; \, H^{\frac{3}{2}}_x(\R)\right)$.
\end{thm}

\bigskip
\noindent{\bf Acknowledgements.} Y. Ran and S. M. Sun were partly supported by National Science Foundation under grant No. DMS-1210979.  B.-Y. Zhang was partially supported by grants from  the Simons Foundation (201615) and NSF  of China (11231007, 11571244). The paper is based upon the first part of Y. Ran's Ph.D. dissertation at Virginia Polytechnic Institute and State University.

\end{document}